\documentclass[12pt]{amsart}

\usepackage{amsxtra,amssymb,amsthm,amsmath,amscd,url,amsrefs}
\usepackage{yhmath}

\usepackage{enumitem}
\newlist{enumth}{enumerate}{1}
\setlist[enumth]{label=\emph{(\arabic*)}, ref=(\arabic*)}

\usepackage[utf8]{inputenc}
\usepackage{eucal}
\usepackage{fullpage}
\usepackage{mathrsfs}
\usepackage{csquotes}
\usepackage[colorlinks]{hyperref}
\usepackage[]{graphicx}
\usepackage{subcaption}
\usepackage{dsfont}
\usepackage{float}
\usepackage{tikz-cd}

\renewcommand{\mathcal}{\mathscr}

\def\setminus{\mathchoice
	{\mathbin{\vrule height .62ex width 1.61ex depth -.38ex}}
	{\mathbin{\vrule height .62ex width 1.61ex depth -.38ex}}
	{\mathbin{\vrule height .50ex width 0.85ex depth -.28ex}}
	{\mathbin{\vrule height .20ex width 0.570ex depth -.24ex}}
}

\DeclareMathSymbol{A}{\mathalpha}{operators}{`A}%
\DeclareMathSymbol{B}{\mathalpha}{operators}{`B}%
\DeclareMathSymbol{C}{\mathalpha}{operators}{`C}%
\DeclareMathSymbol{D}{\mathalpha}{operators}{`D}%
\DeclareMathSymbol{E}{\mathalpha}{operators}{`E}%
\DeclareMathSymbol{F}{\mathalpha}{operators}{`F}%
\DeclareMathSymbol{G}{\mathalpha}{operators}{`G}%
\DeclareMathSymbol{H}{\mathalpha}{operators}{`H}%
\DeclareMathSymbol{I}{\mathalpha}{operators}{`I}%
\DeclareMathSymbol{J}{\mathalpha}{operators}{`J}%
\DeclareMathSymbol{K}{\mathalpha}{operators}{`K}%
\DeclareMathSymbol{L}{\mathalpha}{operators}{`L}%
\DeclareMathSymbol{M}{\mathalpha}{operators}{`M}%
\DeclareMathSymbol{N}{\mathalpha}{operators}{`N}%
\DeclareMathSymbol{O}{\mathalpha}{operators}{`O}%
\DeclareMathSymbol{P}{\mathalpha}{operators}{`P}%
\DeclareMathSymbol{Q}{\mathalpha}{operators}{`Q}%
\DeclareMathSymbol{R}{\mathalpha}{operators}{`R}%
\DeclareMathSymbol{S}{\mathalpha}{operators}{`S}%
\DeclareMathSymbol{T}{\mathalpha}{operators}{`T}%
\DeclareMathSymbol{U}{\mathalpha}{operators}{`U}%
\DeclareMathSymbol{V}{\mathalpha}{operators}{`V}%
\DeclareMathSymbol{W}{\mathalpha}{operators}{`W}%
\DeclareMathSymbol{X}{\mathalpha}{operators}{`X}%
\DeclareMathSymbol{Y}{\mathalpha}{operators}{`Y}%
\DeclareMathSymbol{Z}{\mathalpha}{operators}{`Z}%

\renewcommand{\leq}{\leqslant}
\renewcommand{\geq}{\geqslant}

\newcommand{\Cc}{\mathbf{C}}

\newcommand{\Nn}{\mathbf{N}}

\newcommand{\Zz}{\mathbf{Z}}

\newcommand{\Rr}{\mathbf{R}}

\newcommand{\Ss}{\mathbf{S}}

\newcommand{\Qq}{\mathbf{Q}}

\newcommand{\tp}{\mathcal{N}(T,t_q)}

\newcommand{\J}{\text{J}_0}

\newcommand{\mods}[1]{\,(\mathrm{mod}\,{#1})}

\newcommand{\ind}{\mathds{1}}

\renewcommand{\hat}{\widehat}

\DeclareMathOperator{\wass}{\mathcal{W}}

\DeclareMathOperator{\li}{li}

\DeclareMathOperator{\supp}{supp}

\renewcommand{\rho}{\varrho}

\DeclareMathSymbol{\gena}{\mathord}{letters}{"3C}
\DeclareMathSymbol{\genb}{\mathord}{letters}{"3E}

\theoremstyle{plain}
\newtheorem{theorem}{Theorem}[section]
\newtheorem{lemma}[theorem]{Lemma}
\newtheorem{corollary}[theorem]{Corollary}
\newtheorem{conjecture}[theorem]{Conjecture}

\newtheorem{proposition}[theorem]{Proposition}

\theoremstyle{remark}

\theoremstyle{definition}

\newtheorem{remark}[theorem]{Remark}

\renewcommand{\geq}{\geqslant}
\renewcommand{\leq}{\leqslant}
\renewcommand{\le}{\leqslant}
\renewcommand{\ge}{\geqslant}
\renewcommand{\Re}{\mathfrak{Re}\,}
\renewcommand{\Im}{\mathfrak{Im}\,}

\newcommand{\E}{\mathds{E}}
\newcommand{\Var}{\mathds{V}}
\newcommand{\rad}{\mathrm{rad}}

\begin{document}

\title{A Wasserstein metric approach to generalized Skewes' numbers. I. Prime number races}

\author{A. Bailleul}
\address{ENS Paris-Saclay, Centre Borelli, UMR 9010, 91190 Gif-sur-Yvette, France}
\email{alexandre.bailleul@ens-paris-saclay.fr}

\author{M. Hayani}
\address{Max Planck Institute for Mathematics Vivatsgasse 7 53111, Bonn, Germany}
\email{hayani@mpim-bonn.mpg.de}

\author{T. Untrau}
\address{Univ Rennes, CNRS, IRMAR - UMR 6625, F-35000 Rennes, France} 
\email{theo.untrau@ens-rennes.fr}

\date{} 



\begin{abstract}
We study generalized Skewes' numbers, which are the locations of the first sign change between two comparable prime counting functions. In the context of the race between quadratic residues and quadratic nonresidues, we construct sequences of highly composite moduli $q$ such that those Skewes' numbers grow very rapidly in some sense. This disproves unconditionally a conjecture of Fiorilli. In the other direction, assuming the Generalized Riemann Hypothesis and an effective linear independence hypothesis, we establish conditional upper bounds for generalized Skewes' numbers. Our approach relies on a quantitative Kronecker-Weyl theorem formulated in terms of the $1$-Wasserstein metric to obtain explicit rates for the convergence to the limiting distributions in these races.
\end{abstract}

\maketitle
\section{Introduction}

\subsection{Chebyshev's bias}

Chebyshev's bias is the phenomenon that, for most $x \geq 2$, among prime numbers smaller than $x$, there are more that are congruent to 3 modulo 4 rather than 1 modulo 4. This phenomenon has been extensively studied, especially after the seminal work of Rubinstein and Sarnak \cite{RS94}. In particular, they showed that this phenomenon extends to more general \enquote{prime number races}, in which one compares the number of primes $p \equiv a \mods{q}$ and $p \equiv b \mods{q}$ smaller than $x$. In this context, there is a bias towards the congruence $a$ when $b$ is a quadratic residue modulo $q$ and $a$ is not, while there is no bias if both $a$ and $b$ are simultaneously squares or nonsquares modulo $q$. Precisely, if we denote by $\pi(x;q,a)$ the number of primes $p$ up to $x$ that satisfy $p \equiv a \mods{q}$, they proved that the set
\[ \mathcal{P}_{q;a,b} := \{ x \geqslant 2 \mid \pi(x;q,a) > \pi(x;q,b) \} \]
admits a positive logarithmic density $\delta_{q;a,b}$, meaning that 
\begin{equation} \label{eq: conv_delta_q_a_b}
\lim_{Y \to \infty} \left| \frac{1}{\log(Y)} \int_{2}^{Y} \ind_{\mathcal{P}_{q;a,b}}(y) \frac{\mathrm{d}y}{y}  -  \delta_{q;a,b} \right|=\lim_{X \to \infty} \left| \frac{1}{X} \int_{\log 2}^{X} \ind_{\mathcal{P}_{q;a,b}}(e^y) \mathrm{d}y  -  \delta_{q;a,b} \right| =0,
\end{equation}
and determined the conditions under which this density is smaller than, equal to, or greater than $\frac12$. Their results are conditional on the Generalized Riemann Hypothesis (GRH) for Dirichlet $L$-functions modulo $q$, and also on a linear independence hypothesis which they called GSH, for Grand Simplicity Hypothesis (later authors called this hypothesis LI for Linear Independence, and we will instead use this second terminology). The article of Rubinstein and Sarnak also studies more general prime number races, for instance the race between all quadratic residues and all nonresidues. More precisely, denoting by $R_q := \{ a^2 : \ a \in (\Zz/ q \Zz)^{\times} \}$ and by $NR_q :=   (\Zz/ q \Zz)^{\times} \setminus R_q $, we can ask how often (in logarithmic density) does the following inequality hold:
\begin{equation} \label{eq: teamed up squares} \frac{1}{|R_q|} \sum_{a\in R_q} \pi(x;q,a) > \frac{1}{|NR_q|} \sum_{b\in NR_q} \pi(x;q,b)  \text{ ?} 
\end{equation}

\begin{remark}
In what follows we will mostly use the second formulation of the convergence to the logarithmic density in \eqref{eq: conv_delta_q_a_b} and analogous results (meaning that we include the exponential change of variable in the definition of the integral that should converge to the density).
\end{remark}

\subsection{Generalized Skewes' numbers}

Skewes' number is the smallest number $x_0 \geq 2$ for which the classical prime-counting function $\pi(x)$ exceeds the logarithmic integral $\li(x)=\int_0^x \tfrac{\mathrm{d}t}{\log t}$ (defined as the Cauchy principal value). Littlewood~\cite{Littlewood} proved in $1914$, via the classical explicit formula for $\pi(x)-\li(x)$, that this difference changes sign infinitely often, and therefore that such an $x_0$ exists, but without an effective bound for it. In 1933, assuming the Riemann Hypothesis (RH), Skewes~\cite{Skewes} used the same explicit formula together with theoretical bounds on the distribution of the zeta function zeros and Diophantine approximations to show that $x_0<\exp(\exp(\exp(79)))$ (a triple exponential bound). He later obtained an unconditional bound as well in \cite{skewesII}. Later work, beginning with Lehman~\cite{Lehman} in 1966, introduced a smoothed variant of $\pi(x)-\li(x)$, defined as a short average of this difference weighted by an exponentially decreasing smooth function. Lehman then derived an explicit formula for this smoothed quantity, in which the Gaussian weight strongly suppresses the contribution of high zeros. This makes the main term depend mainly on the low--lying zeros. By computing numerically these zeros to high accuracy, Lehman and subsequent authors (te Riele~\cite{Riele}, Bays and Hudson~\cite{BaysHudson}) reduced this bound to $x_0< 1.39 \times 10^{316}$. It is also known that $x_0 > 10^{19}$ by direct calculations due to Büthe \cite{buthe}, but this lower bound remains far from the expected order of magnitude of $x_0$ (which is widely believed to be close to Bays and Hudson's $10^{316}$ upper bound).\\
In the case of prime number races, as explained by Fiorilli~\cite{Fiorilli}, one might consider 
\[ x_{q;a,b} := \inf\left\{x\ge 2\, :\, \pi(x;q,a) > \pi(x;q,b) \right\} \]
for a race where $a$ is a quadratic residue modulo $q$ and $b$ is not. In that case, the work of Rubinstein and Sarnak shows that the inequality $\pi(x;q,a) > \pi(x;q,b)$ only takes place for a small proportion of $x$ (with respect to the logarithmic density), so that the number $x_{q;a,b}$ can be considered as a generalized Skewes' number, as it is defined as the first realization of a \enquote{rare event} in a prime number race. Similarly, for the race between all quadratic residues and all nonresidues, one can define the associated generalized Skewes' number as
\[ x_{q;R,NR} := \inf\left\{x\ge 2\, \text{ such that } \eqref{eq: teamed up squares} \text{ holds}  \right\}  \]
(and $x_{q;NR,R}$ for the analogous quantity when the inequality in \eqref{eq: teamed up squares} is reversed). In this setting, Fiorilli made the following conjecture concerning the growth of $x_{q;R,NR}$.
\begin{conjecture}[{\cite[Conj. 1.19]{Fiorilli}}] \label{conj: fiorilli R vs NR} For all integers $q$, denote by $\rho(q)$ the number of square roots of $1$ modulo $q$. If $(q_n)_{n \geqslant 1}$ is a sequence of integers such that
 $\frac{\rho(q_n)}{\log \rad(q_n)}\to \infty$, then we have \[ \log \log x_{q_n;R, NR} \asymp \frac{\rho(q_n)}{\log \rad(q_n)} .\]
\end{conjecture}


\subsection{Overview of the results}

In this paper, we first show that the Skewes' numbers for the race between quadratic residues and nonresidues can actually grow much faster than what was predicted in Conjecture \ref{conj: fiorilli R vs NR}. 

\begin{theorem} \label{th: fiorilli_disproof}
Let $h\, :\, \Nn\to \Rr_{>0}$ be an increasing function tending to infinity. Then, for all $L>0$, there exist sequences $(q_n)_n$ and $(q_n')_n$ of square-free integers such that both $\rho(q_n)/\log q_n$ and $\rho(q_n')/\log q_n'$ tend to $\infty$ as $n \to \infty$, and such that for all sufficiently large $n$, we have \[ h (x_{q'_n; R, NR})>L\,  \frac{\rho(q_n')}{\log q_n'}\quad \text{and}\quad h (x_{q_n;NR,R})>L\,  \frac{\rho(q_n)}{\log q_n} \, .\]
\end{theorem}

In particular, neither of the inequalities of Fiorilli's conjecture hold in general. The construction is an adaptation of a method of Chowla used in the study of the least quadratic nonresidue modulo a prime. We will see in Theorem \ref{skewes-qrnr} that the additional quantity $\log \rho(q_n)$ is crucial in bounding $x_{q_n; R, NR}$. However, we expect that the large deviation heuristics used by Fiorilli to state his conjecture should still hold for some sequences $(q_n)_n$ of highly composite numbers, for example the primorials $q_n = p_1 \dots p_n$ where $(p_n)_n$ is the increasing sequence of primes.

Then, we focus on the question of finding upper bounds for the generalized Skewes' numbers. A natural approach consists in making the convergence \eqref{eq: conv_delta_q_a_b} quantitative, since this will allow us to determine an explicit $Y_0$ such that for all $Y \geqslant Y_0$, the integral
\[ \int_{2}^{Y} \ind_{\mathcal{P}_{q;a,b}}(y) \frac{\mathrm{d}y}{y}\]
is positive, hence there exists $y \leqslant Y_0$ such that $\pi(y;q,a) > \pi(y;q,b)$. Inspired by recent work of Lamzouri \cite{lamzouri_eli} and Ng \cite{Ng-ELI}, we carry out this approach by introducing an effective form of the LI conjecture, which we call $\mathrm{ELI_A}$ (see Conjecture \ref{conj: ELI_A}), and which states that the non-trivial linear combinations of positive imaginary parts $\gamma_j$ of zeros of Dirichlet $L$-functions satisfy $$\left|\sum_{j=1}^{N(T)} m_j \gamma_j\right| \gg_A N(T)^{-N(T)^A},$$ 
where $N(T)$ is the number of $\gamma_j$ such that $0 < \gamma_j \leqslant T$, and the coefficients are integers such that $|m_j| \leqslant N(T)$. For our applications, it is often relevant to restrict to a subset $\mathcal{X}_q$ of the set of all Dirichlet characters modulo $q$, so we actually formulate a conjecture $\mathrm{ELI}_A(\mathcal X_q)$.

Under this assumption, we obtain a quantitative form of \eqref{eq: conv_delta_q_a_b}: 
\begin{theorem} \label{cor: effective rate a vs b}
 Let $q\ge 3$ and let $a,b$ be distinct invertible residue classes modulo $q$. Assume $\mathrm{GRH}$ and $\mathrm{ELI_A}(\mathcal{X}_q)$ for some $A>1$ and $\mathcal{X}_q$ equal to the support of the Fourier transform of $t_q = \varphi(q)(\ind_{\{a\}} - \ind_{\{b\}})$. Then, uniformly for $$X\ge (\varphi (q) \log q )^{(\mathcal{L} \varphi (q) \log q) ^A}, $$ 
where $\mathcal{L}>0$ is an absolute effective constant, we have
    \[\left| \frac{1}{X-\log 2} \int_{\log 2}^{X} \ind_{\mathcal{P}_{q;a,b}}(e^y) \mathrm{d}y  -  \delta_{q;a,b} \right| \ll_{A} \varphi(q)^{\tfrac{3}{2}} (\log q)^{\tfrac{3}{4}+\tfrac{1}{4A}}(\log \log X) (\log X) ^{-\tfrac{1}{4A}}  .\]
\end{theorem}

 As a corollary, we obtain a double exponential bound for the generalized Skewes' number $x_{q;a,b}$.

\begin{theorem}\label{skewes-primerace}
 Let $q \geq 3$ and let $a,b$ be distinct invertible residue classes modulo $q$. Under the assumptions of the Theorem $\ref{cor: effective rate a vs b}$, we have \[\log \log \log x_{q;a,b} \ll_A \log \varphi(q).\]
\end{theorem}

We note that Schlage-Puchta studied the related question of finding an upper bound for the first $x$ for which
\[
\pi(x;q,1) > \max_{a \in (\Zz / q \Zz)^{\times} \setminus \{1 \} } \pi(x;q,a).
\]
In \cite[Th. 1]{schlage-puchta_sign_changes}, he obtained under GRH that there exists such an $x$ satisfying
\[\log \log x < (q^{+})^{170} + \exp(18 \rho(q)),\]
where $q^+ = \max(q, \exp(1260))$, and also gave a lower bound on the number of sign changes of $\pi(t;q,1) - \max_{a \neq 1 \mods q } \pi(t;q,a)$ in the range $2 \leqslant t \leqslant x$.

For simplicity, we only stated Theorem \ref{cor: effective rate a vs b} for the sign changes of $\pi(x;q,a)- \pi(x;q,b)$, but we also prove similar bounds for more general prime counting functions (denoted $\pi(x;t_q)$ in Section \ref{sec: prime-races}). In particular, we can also find an upper bound for the Skewes' number $x_{q;R,NR}$ associated with the race between all quadratic residues and all nonresidues. In that case, our result is the following.
\begin{theorem} \label{skewes-qrnr}
Let $(q_n)_n$ be a sequence of integers such that
 $\frac{\rho(q_n)}{\log \rad (q_n)}\to \infty$. Assume $\mathrm{GRH}$ and $\mathrm{ELI}_A(\mathcal{X}_{q_n})$ for some $A>1$ and $\mathcal{X}_{q_n}$ equal to the set of non-trivial quadratic Dirichlet characters modulo $q_n$. Then we have \[ \log \log \log x_{q_n;R,NR}  \ll_A \frac{\rho(q_n)}{\log \rad(q_n)}+ \log \rho \bigl(q_n\bigr) .\]
\end{theorem}

One of the main tools to prove such quantitative rates of convergence is the $1$-Wasserstein metric and an inequality due to Bobkov-Ledoux akin to an Erd\H os--Turán inequality to obtain a quantitative version of the Kronecker-Weyl equidistribution theorem, as explained in Section \ref{sec: kronecker-weyl}. A key component of our proofs is to make explicit the dependencies on $q$ in some classical results going back to Rubinstein-Sarnak's work \cite{RS94}.

We note that the shape of the lower bound we propose in Conjecture \ref{conj: ELI_A} is just a working hypothesis, but it could be replaced by any effective lower bound on linear combinations of the $\gamma_j$, and our method would still give rates of convergence to the relevant limiting distributions.

The methods developed in this paper admit natural unconditional generalizations in the \enquote{function field case}, where instead of counting prime numbers we count irreducible polynomials over finite fields. This will be detailed in a follow-up to this paper.

\subsection*{Organization of the paper} In Section \ref{sec : disproof_fiorilli} we provide a construction of sequences of integers that allow us to disprove both inequalities of Fiorilli's conjecture. In Section \ref{sec: kronecker-weyl}, we recall some fundamental properties of Wasserstein metrics and prove a quantitative version of the Kronecker--Weyl Theorem with respect to $\wass_1$. In Section \ref{sec: lin-indep}, we formulate an effective linear independence hypothesis for the positive imaginary parts of zeros of Dirichlet $L$-functions modulo $q$. In Section \ref{sec: prime-races}, we use the aforementioned hypothesis and bounds on Wasserstein metrics to obtain an effective rate of convergence in the context of prime number races. This also allows us to obtain the first (conditional) bounds on Skewes numbers for some of those races. 

\subsection*{Notations} 
\begin{itemize}
\item $\rad(q)$ denotes the radical of an integer $q$.

\item $\rho(q)$ denotes the number of $x \in (\Zz / q \Zz)^{\times}$ such that $x^2 \equiv 1 \mods{q}$.

\item $t_q$ will denote a map from $(\Zz / q \Zz)^{\times}$ to $\Rr$, orthogonal to the principal Dirichlet character $\chi_0$ modulo $q$ (we recall the definition of the inner product and of the Fourier transform at the beginning of Section \ref{sec: prime-races}).

\item $t_q^*:(\Zz/q\Zz)^\times \to \Rr$ is given by $t_q^*(a)=t_q(a^2)$ for all $a\in (\Zz/q\Zz)^\times$.
\item $\lambda(t_q)$ denotes the $L^1$ norm of $\widehat{t_q}$, \emph{i.e.}  $\lambda(t_q):=\sum_{\chi} \bigl| \langle t_q,\chi \rangle \bigr|$, where the sum ranges over multiplicative characters $\chi$ modulo $q$.

\item $C(t_q)=\max(\lambda(t_q)(\log q)^2,\lambda(t_q^*)\log q)$ 
\item $k(t_q)$ is the real number defined by the equality $\log k(t_q) = \frac{1}{|\supp(\widehat{t_q})|} \sum_{\chi \in \supp(\widehat{t_q})} \log q_{\chi},$ where $q_{\chi}$ denotes the conductor of $\chi$.

\item $(\gamma_n)_{n \geq 1}$ is an enumeration in non-decreasing order of the positive imaginary parts of the non trivial zeros, counted with multiplicity, of the $L$-functions attached to Dirichlet characters in $\supp(\widehat{t_q})$. 

\item With the previous notations, $N(T,t_q)$ denotes the number of indices $n$ such that $0 < \gamma_n \leq T$.

\item $\mathrm{ELI}_A(t_q)$ denotes the conjecture $\mathrm{ELI}_A(\mathcal{X}_q)$ (Conjecture \ref{conj: ELI_A}) when $\mathcal{X}_q = \mathrm{supp}(\widehat{t_q})$.

\end{itemize}
Due to the definition of the logarithmic density, we are often led to think of functions defined on $\Rr_{>0}$ as random variables defined on the (varying) probability space $[\log 2, X]$ endowed with its renormalized Lebesgue measure. With this point of view in mind,
\begin{itemize}
\item $\mu_X$ is defined in \eqref{eq: f de E} as the distribution of the random variable $E$ defined in \eqref{Def-E(y)}.
\item $\mu_X^{(T)}$ is defined in \eqref{eq: f de ET} as the distribution of the random variable $E^{(T)}$ defined in \eqref{Def-ET(y)}.
\item $\nu_X^{(T)}$ is defined in \eqref{eq: def nuXT} as the distribution of the random variable $y \mapsto \left( e^{i \gamma_1 y}, \dots, e^{i \gamma_{N(T,t_q)} y} \right)$ (which takes values in $(\Ss^1)^{N(T,t_q)}$).
\item $\mu^{(T)}$ is defined in \eqref{eq: def mu^T} as the pushforward measure of $\lambda_T$ (which is the Haar probability measure on $\left( \Ss^1 \right)^{N(T, t_q)}$) via the map $g^{(T)}$ defined in \eqref{def gT}.
\end{itemize}

\subsection*{Acknowledgements} We would like to thank Bence Borda for answering our questions on \cite{borda_cuenin}. We also thank Lucile Devin, Daniel Fiorilli, Florent Jouve, Emmanuel Kowalski and Youness Lamzouri for helpful discussions to improve this paper.

\section{A disproof of Fiorilli's conjecture} \label{sec : disproof_fiorilli}
For a prime $p\ge 3$, let $n_p\ge 1$ denote the least quadratic nonresidue modulo $p$. The integer $n_p$ has been studied by several authors. Friedlander~\cite{Frid}, Salié~\cite{Salié} and Chowla \cite[Th. 3.10]{bruni} showed independently that there are infinitely many primes $p$ for which $n_p\gg \log p$. 
 This result was improved by Graham and Ringrose~\cite{GrRi} who proved that there are infinitely many primes for which $n_p\gg \log p (\log \log \log p)$. In terms of upper bounds, the best unconditional result is due to Burgess \cite{burgess}, which states that $n_p \ll_{\varepsilon} p^{\frac{1}{4 \sqrt{e}} + \varepsilon}$ for all $\varepsilon >0$.

The integer $n_p$ was also studied conditionally on GRH, for instance Montgomery~\cite{Montg} proved that there are infinitely many primes for which $n_p \gg \log p \log \log p$, and Ankeny~\cite{Ank} proved that $n_p\ll (\log p)^2$. In this paper, we are interested in the least primes that are quadratic residues and nonresidues modulo a composite modulus $q$. For $q\ge 3$, let $\Phi(q)$ (resp. $\Psi(q)$) denote the least prime that is a quadratic residue (resp. nonresidue) modulo $q$. We prove the following result:
\begin{theorem}\label{lqr-nr}
Let $f\colon \Nn\to \Rr_{>0}$ be an increasing function tending to infinity, such that $\log f(n)\leq n^{n/3}$. Then, there exist sequences $(q_n)_n$ and $(q_n')_n$ of square-free integers such that \begin{align}
     \log q_n\sim \log q_n' \asymp n (\log n)(\log n f(n) )\quad &\text{as}\quad n\to \infty, \label{q-size}\\
    \frac{\rho(q_n')}{\log q_n'}\asymp \frac{\rho(q_n)}{\log q_n}\asymp f(n) \quad &\text{as}\quad n\to \infty, \label{quotient-size}\\
    \min (\Psi(q_n),\Phi(q_n') )\gg n \log n\quad &\text{as}\quad n\to \infty. \label{Chowla-bound}
\end{align}
\end{theorem}
Our strategy relies on Chowla's argument as presented in \cite{bruni} to construct the prime factors of $q_n$, which will lead to a bound of type~\eqref{Chowla-bound}, but in order to obtain~\eqref{q-size} and~\eqref{quotient-size} we need to choose these prime factors to have a specific uniform size.
\begin{proof}
Let $n$ be a sufficiently large integer, and let $p_1,\dots,p_n$ be the first $n$ odd primes. Denote $Q_n:=8p_1\cdots p_n$ and \[S_n:=\left\{(x_1,\dots,x_n)\in \prod_{i=1}^n \bigl(\Zz/p_i\Zz)^\times\, :\, x_i\text{ is a quadratic residue modulo }p_i\ (1\le i \le n)\right\} \, .\]
The Chinese remainder theorem shows that for $\xi=(x_1,\dots,x_n)\in S_n$ there exists an integer $a_\xi$ such that \[ a_\xi \equiv 1 \mods{8} \quad\text{and}\quad a_\xi \equiv x_i \mods{p_i} \quad (1\le i\le n)\, .\]
Linnik's theorem shows that there exists a prime $\mathfrak{p}_\xi\equiv a_\xi \mods{Q_n}$ such that $\log \mathfrak{p}_\xi \ll \log Q_n$.
Now, if we take $q_n$ to be a product of $\mathfrak{p}_\xi$ for distinct $\xi \in S_n$, we obtain ~\eqref{Chowla-bound} immediately; indeed, by quadratic reciprocity, we have that $2,p_1,\dots,p_n$ are all quadratic residues modulo all factors of $q_n$, thus they are all quadratic residues modulo $q_n$ by the Chinese remainder theorem. This implies that $\Psi(q_n)>p_n\gg n\log n $ by the prime number theorem.\\
Define \[\mathcal{S}_n:=\left\{\xi \in S_n\, :\, \mathfrak{p}_\xi>\sqrt{Q_n} \right\}\, ,\]
and $ s_n:= |S_n|=\prod_{i=1}^n \frac{p_i-1}{2} $. The prime number theorem implies that \[\log s_n=\sum_{i=1}^n \log p_i+\sum_{i=1}^n \log \left(\frac{p_i-1}{2p_i}\right) \sim p_n\sim n\log n\quad (n\to \infty) \, ,\]
and that $\log Q_n \sim n\log n$. Therefore, for sufficiently large $n$, we have $s_n>2\sqrt{Q_n}$, which implies $|\mathcal{S}_n|>\sqrt{Q_n}$ since $\xi \mapsto \mathfrak{p}_{\xi}$ is injective. Since the function $x\mapsto 2^x/x$ is continuous and increasing in $(2,\infty)$, there exists a unique $x_n \in (2,\infty)$ such that \[\frac{2^{x_n}}{x_n}=n (\log n) f(n) \, ,\] assuming $n$ is large enough.
We have in particular that $x_n\sim \log (n f(n))/\log 2$ as $n \to \infty$. Denote $r_n=\lfloor x_n \rfloor$. Our assumption on $f$ implies that \[\log \log( n f(n)) \leq \left(\frac{n}{3}+1\right)\log n  \, .\] Since $n$ is sufficiently large, we have $r_n< \sqrt{Q_n}$. Thus, we can choose $r_n$ distinct elements $\xi_1,\dots,\xi_{r_n} \in \mathcal{S}_n$. Define
$ q_n:=\mathfrak{p}_{\xi_1} \cdots \mathfrak{p}_{\xi_{r_n}}$. 
We have \[\log q_n =\sum_{i=1}^{r_n}\log \mathfrak{p}_{\xi_i} \asymp r_n \log Q_n\asymp n \log n \log (n f(n)) \, ,\]
and \[ \frac{\rho(q_n)}{\log q_n}\asymp \frac{2^{r_n}}{r_n n\log n} \asymp \frac{2^{x_n}}{x_n (n\log n)}=f(n) \, .\]
This proves the statement for $q_n$. Let $\xi'=(y_1,\dots,y_n)\in \prod_{i=1}^n \bigl(\Zz/p_i\Zz)^\times $ such that, for all $1\le i \le n$, $y_i$ is a quadratic nonresidue modulo $p_i$. By the Chinese remainder theorem there exists $b\in \Zz$ such that \[ b \equiv 5\mods{8} \quad\text{and}\quad b\equiv y_i \mods{p_i} \quad (1\le i \le n)\, .\]
By Linnik's theorem, there exists a prime $\mathfrak{q}$, such that $\mathfrak{q}\equiv b \mods{Q_n}$ and $\log \mathfrak{q}\ll \log Q_n \ll n\log n$. Define $q_n':=\mathfrak{q}q_n$. We have $\log q_n'=\log \mathfrak{q}+\log q_n \sim \log q_n$, which proves \eqref{q-size} and \eqref{quotient-size} for $q_n'$. Moreover, applying quadratic reciprocity again, we deduce \[ \left( \frac{2}{q_n'} \right)=\left(\frac{2}{\mathfrak{q}}\right)\prod_{i=1}^{r_n}\left(\frac{2}{\mathfrak{p}_{\xi_i}}\right)=-1\, ,\]
and for $1\le i\le n$,
\[\left( \frac{p_i}{q_n'} \right)=\left(\frac{p_i}{\mathfrak{q}}\right)\prod_{i=1}^{r_n}\left(\frac{p_i}{\mathfrak{p}_{\xi_i}}\right)=-1\, .\]
This proves that $2,p_1,\dots,p_n$ are all quadratic nonresidues modulo $q_n'$. Thus, $\Phi(q_n')\gg n\log n$, which finishes the proof of the Theorem.
\end{proof}

As a consequence of these constructions, we can derive our result disproving Fiorilli's conjecture (Conjecture \ref{conj: fiorilli R vs NR}), which we stated in the introduction. 
\begin{proof}[Proof of Theorem $\ref{th: fiorilli_disproof}$]
We may assume that $\log \log h(n)\leq n^{n/3}$, since we can otherwise replace $h(n)$ with $\min(h(n),\exp(\exp(n^{n/3})))$.
Let us assume for the sake of contradiction that there exists $L>0$ such that for all sequences $(q_n)_n$ of square-free integers satisfying $\rho(q_n)/\log q_n \to \infty$ as $n \to \infty$, we have  $h (x_{q_n;R,NR})\le L\, \rho(q_n)/\log q_n$ for infinitely many $n$. 
Consider $f(n)=\log h(n)$. Let $(q_n')_n$ be a sequence provided by Theorem~\ref{lqr-nr}. By~\eqref{Chowla-bound}, we have for $n$ sufficiently large
\[ h(x_{q_n';R,NR})\ge h(\Phi(q_n')) > h(n) \, .\]
Moreover, by~\eqref{quotient-size}, we have for infinitely many $n$ \[ h (x_{q_n';R,NR})\le L\, \frac{\rho(q_n')}{\log q_n'} \ll f(n)=\log h(n) \, .\]
This provides a contradiction, which proves the statement for $x_{q;R,NR}$. The statement regarding $x_{q;NR,R}$ is proved similarly using $\Psi$ instead of $\Phi$ in Theorem~\ref{lqr-nr}. 
\end{proof}
Taking $h(n)=\log \log n$, we obtain a sequence contradicting Fiorilli's conjecture. Taking $h(n)=\log \log \log n$ shows that the bound in Theorem~\ref{skewes-qrnr} would not hold without the term $\log \rho(q_n)$. 

\section{A quantitative Kronecker--Weyl theorem with respect to the $1$--Wasserstein metric} \label{sec: kronecker-weyl}

Among distances between probability measures that allow one to quantify the weak convergence, Wasserstein metrics have the advantage of being defined on any \emph{Polish space} (\emph{i.e.} a separable and complete metric space), and of being compatible with Lipschitz maps, which makes them very convenient to keep track of rates of convergence when working with pushforward measures.
Moreover, they proved to be convenient to prove rates of convergence for \enquote{degenerate} measures, meaning measures supported on a submanifold of a given manifold, while distances such as the ball discrepancy or the box discrepancy \enquote{might not see the submanifold}, as they may assign mass zero to it. This situation appears very naturally in the context of the Kronecker--Weyl theorem, where the presence of $\Qq$-linear relations among real numbers $\gamma_1, \dots, \gamma_N$ can prevent the sequence $((e^{in \gamma_1}, \dots, e^{in \gamma_N}))_{n \geqslant 1}$ from equidistributing in the full torus $(\Ss^1)^N$. To obtain quantitative rates of convergence towards the uniform measure on a suitable subtorus, one is naturally led to look for an intrinsic notion of distance between measures, that is: one that does not use the fact that the subtorus arises as a subset of $(\Ss^1)^N$ (apart from the definition of the metric). The Wasserstein metrics $\wass_p$ provide a solution to this problem, because they are only defined in terms of the metric on the subtorus (or, in other words, in terms of test functions defined on the subtorus). Moreover, they satisfy Fourier analytic inequalities similar to the classical Erd\H os--Tur\'an inequality, allowing us to prove a quantitative version of the Kronecker--Weyl theorem.

\subsection{Properties of Wasserstein metrics} 

Let $(M,d)$ be a Polish space and let $p \geqslant 1$. If $\mu$ and~$\nu$ are two Borel probability measures on $M$ (since $M$ is assumed to be Polish, this automatically implies that they are Radon measures), we let
$\Pi(\mu,\nu)$ be the set of probability measures on $M\times M$ with
marginals $\mu$ and~$\nu$ (\emph{i.e.} $\pi(A \times M) = \mu(A)$ and $\pi(M \times A) = \nu(A)$ for all Borel subsets $A \subset M$). Such probability measures on $M \times M$ are also called \emph{couplings} of $\mu$ and $\nu$. Then, the $p$-Wasserstein distance between~$\mu$ and~$\nu$ is defined by
\begin{equation} \label{eq: def wasserstein p}
    \wass_p(\mu,\nu)=\inf_{\pi\in \Pi(\mu,\nu)}\Bigl(
\int_{M \times M}d(x,y)^p\mathrm{d}\pi(x,y)\Bigr)^{1/p}.
\end{equation}
This definition can be interpreted as minimizing the transport cost from a certain distribution of piles of sand on $M$ to another distribution, when the cost of transportation of one unit of sand from $x$ to $y$ is $d(x,y)^p$. We refer to \cite{villani} for a thorough introduction to the vast subject of optimal transport and its history. Let us also mention that the terminology \emph{Monge--Kantorovich distance of order $p$} also appears in the literature to refer to $\wass_p$, and $\wass_1$ is often referred to as the \emph{Kantorovich--Rubinstein distance}.

From now on, all probability measures are implicitly assumed to be Borel. A very important feature of $\wass_p$ is that it is indeed a metric, and that it induces the weak convergence of measures.

\begin{proposition} \label{prop: wass properties} Let $(M,d)$ be a Polish space.
      \begin{enumth}
		\item For all $p \geqslant 1$, $\wass_p$ is a metric on the space $\mathcal P_p(M)$ of probability measures on~$M$ having a finite moment of order $p$ (probability measures $\mu$ on $M$ such that for some (and thus any) $x_0 \in M$, $ \displaystyle\int_{M} d(x,x_0)^p \mathrm{d}\mu(x) < \infty$).
		
		\item Given a sequence $(\mu_n)_{n \geqslant 1}$ in $\mathcal P_p(M)$ and a probability measure $\mu$ on $M$, we have $\wass_p(\mu_n,\mu) \underset{n \to \infty}{\longrightarrow} 0$ if and only if, for all bounded
		continuous functions $f\colon M\to\Cc$, we have
		\[
		\lim_{n\to\infty}\int_M f \mathrm{d}\mu_n=\int_Mf \mathrm{d}\mu.
		\]
		(\emph{i.e.} $\mu_n$ converges weakly to $\mu$) 
		and for some (and thus any) $x_0 \in M$, \[ \lim_{n \to \infty}\int_M d(x, x_0)^p \mathrm{d} \mu_n(x) = \int_M d(x, x_0)^p \mathrm{d} \mu (x) \] (convergence of moments of order $p$).
		
		\item For all $p \geqslant 1$, the metric space $(\mathscr{P}_p(M), \wass_p)$ is again a Polish space. 
		\end{enumth}
\end{proposition}

\begin{proof}
    (1) See e.g. \cite[Th. 7.3]{villani}.

(2) This is \cite[Th. 7.12]{villani}. Note that when $d$ is bounded (in particular, when $M$ is compact), this statement says that $\wass_p$ metrizes weak convergence in the whole space $\mathcal P(M)$ of probability measures on $M$ (since the finite moments assumptions are automatically satisfied).

(3) See e.g. \cite[Th. 6.18]{villani_old_new}. This statement relies on the fact that Cauchy sequences in $(\mathscr{P}_p(M), \wass_p)$ are tight, and on Helly's selection theorem.
\end{proof}

The compatibility with Lipschitz maps that we mentioned above is the following elementary observation: 

\begin{proposition} \label{prop: comp lip}
      Let $(M,d)$ and $(N,\delta)$ be Polish spaces, and $p \geqslant 1$. Let~$c\geq 0$ be a
		real number and let~$f\colon M\to N$ be a $c$-Lipschitz map. For
		any probability measures~$\mu$ and~$\nu$ on~$M$, we have
		\[
		\wass_p(f_*\mu,f_*\nu)\leq c\wass_p(\mu,\nu).
		\]
\end{proposition}

\begin{proof}
See \cite[Th. 1.2 (3)]{ku_wasserstein}
\end{proof}



Finally, the following functional interpretation of $\wass_1$ will be useful throughout this article.

\begin{theorem} \label{th: kantorovich-rubinstein}
    Let $(M,d)$ be a Polish space. For probability measures $\mu$ and $\nu$ on $M$, we have $$\wass_1(\mu,\nu)=\sup_{u}\Bigl|\int_{M}u \mathrm{d}\mu- \int_{M}u \mathrm{d}\nu\Bigr|$$
		where the supremum is over functions~$u\colon M\to \Rr$ which are $1$-Lipschitz and bounded.
\end{theorem}

\begin{proof}
     This is called the Kantorovich--Rubinstein duality theorem, see e.g. \cite[Th. 1.14]{villani}.
\end{proof}




\subsection{Quantitative Kronecker--Weyl} 

We endow $\Ss^1$ with the usual arc-length metric $\ell$, which is defined for all $(\gamma, \theta) \in (-\pi, \pi]^2$ by $\ell(e^{i \gamma}, e^{i \theta}) = \min\{ |\gamma - \theta|, 2\pi - |\gamma - \theta| \}$. Then, for all $N \geqslant 1$, we equip $(\Ss^1)^N$ with the metric $\rho$ defined by
$$
\rho(z,z') = \left( \sum_{j = 1}^{N} \ell(z_j, z'_j)^2 \right)^{1/2}
$$
for all $z = (z_1, \dots, z_N)$ and $z' = (z'_1, \dots, z'_N)$ in $(\Ss^1)^N$. The 1-Wasserstein distance between two probability measures on $(\Ss^1)^N$ is then defined with respect to this metric $\rho$ as in \eqref{eq: def wasserstein p}.

For any closed subgroup $\Gamma$ of $(\Ss^1)^N$, we identify its Haar probability measure $\lambda$ (which is a Borel measure on the topological space $\Gamma$) with the pushforward measure $\iota_* \lambda$ via the canonical embedding $\iota \colon \Gamma \to (\Ss^1)^N$ (which is a Borel measure on $(\Ss^1)^N$). This type of identification does not change the value of Wasserstein distances, thanks to \cite[Th. 1.2 (4)]{ku_wasserstein}.

We recall that the dual of $(\Ss^1)^N$ is isomorphic to $\Zz^N$. Concretely, this means that each character of $(\Ss^1)^N$ is of the form
$$
(z_1, \dots, z_N) \mapsto z_1^{m_1} \cdots z_N^{m_N}
$$
for a unique $m = (m_1, \dots, m_N) \in \Zz^N$. Moreover, if $\mu$ is a Borel probability measure on $(\Ss^1)^N$, its Fourier coefficients are defined for all $m \in \Zz^N$ as 
$$
\widehat{\mu}(m) := \int_{(\Ss^1)^N} z^m \mathrm{d}\mu(z),
$$
 where $z^m$ denotes $z_1^{m_1}\dots z_N^{m_N}$.

\begin{theorem}\label{Q-KW}
	Let $N \geqslant 1$, $\gamma := (\gamma_1, \dots, \gamma_N) \in \Rr^N$, and $z := (e^{i \gamma_1}, \dots, e^{i \gamma_N}) \in (\Ss^1)^N$. We denote by $\Gamma_c$ the closure of the subgroup of $(\Ss^1)^N$ generated by the set $\{ (e^{ix \gamma_1}, \dots, e^{ix \gamma_N}), \ x \in~\Rr\}$, and by $\lambda_c$ its Haar probability measure.  

Let $x_0 > 0$. For all $X > x_0$, denote by $\nu_X$ the probability measure on $(\Ss^1)^N$ defined by
$$
 \int_{(\Ss^1)^N} h \mathrm{d} \nu_X = \frac{1}{X-x_0} \int_{x_0}^{X} h(e^{ix \gamma_1}, \dots, e^{ix \gamma_N}) \mathrm{d} x, \quad \text{for all } h \in \mathcal C((\Ss^1)^N, \Cc).
$$
Then for all $X > x_0$ and all $H \geqslant 1$, 
$$
	\wass_1(\nu_X, \lambda_c) \leqslant \frac{4 \sqrt{3} \sqrt{N}}{H} + \frac{2}{X-x_0} \left(\sum_{\substack{1 \leqslant \|m \|_{\infty} \leqslant H \\ \langle m , \gamma \rangle \neq 0}} \frac{1}{\|m \|_2^2 \langle m, \gamma\rangle^2 }  \right)^{\frac12}\cdot
$$
\end{theorem}

\begin{proof}
The first step is to apply the Bobkov--Ledoux inequality \cite[Eq. (1.6)]{bl_truncated} (also proved independently by Borda \cite[Prop. 3]{borda_bernoulli} with different constants). More precisely, we use the slightly improved form \cite[Th. 1.2 (7)]{ku_wasserstein}, which was indicated as a remark by Bobkov and Ledoux. This gives the following upper bound for all $H \geqslant 1$:
	$$
	\wass_1(\nu_X, \lambda_c) \leqslant \frac{4 \sqrt{3} \sqrt{N}}{H} + \left(\sum_{1 \leqslant \|m \|_\infty \leqslant H} \frac{|\widehat{\nu}_X(m) - \widehat{\lambda}_c(m)|^2}{\|m \|_{2}^2}\right)^{\frac12},
	$$
	where the sum is indexed by vectors $m$ in $\Zz^N$ and $\| m \|_2$ denotes their Euclidean norm.

Now, the Fourier coefficients $\widehat{\nu}_X(m)$ are computed as follows:
\[
	\widehat{\nu}_X(m) = \int_{(\Ss^1)^N} z^m \mathrm{d} \nu_X(z) = \frac{1}{X-x_0}\int_{x_0}^X (e^{i x\gamma_1})^{m_1} \cdots (e^{i x \gamma_N})^{m_N} \mathrm{d} x  = \frac{1}{X-x_0} \int_{x_0}^{X} e^{i \langle m, \gamma \rangle x} \mathrm{d} x.
\]
Thus:
$$
	\widehat{\nu}_X(m) = \begin{cases} 1 \text{ if } \langle m, \gamma \rangle = 0 \\ \frac{e^{i\langle m,\gamma \rangle X } -e^{i\langle m,\gamma \rangle x_0 }}{i  \langle m , \gamma \rangle (X-x_0)} \text{ otherwise}.
		\end{cases}
$$
In particular, in the second case we have the upper bound
\begin{equation} \label{eq: upper bound inverse lin comb}
|\widehat{\nu}_X(m)| \leqslant \frac{2}{|\langle m, \gamma \rangle| (X-x_0)} .
\end{equation} 

Besides, if $\langle m, \gamma \rangle = 0$, the character of $(\Ss^1)^N$ corresponding to $m$ is trivial on elements of the form $(e^{ix \gamma_1}, \dots, e^{ix \gamma_N})$, and therefore is trivial on $\Gamma_c$ by continuity. Conversely, if   $\langle m, \gamma \rangle \neq 0$, then the character corresponding to $m$ is non-trivial on $\Gamma_c$. Therefore, the usual properties of the Haar measure imply that
$$
	\widehat{\lambda}_c(m) = \begin{cases} 1 \text{ if } \langle m, \gamma \rangle = 0 \\ 0 \text{ otherwise}.
\end{cases}
$$
Thus,
$$
	\wass_1(\nu_X, \lambda_c) \leqslant \frac{4 \sqrt{3} \sqrt{N}}{H} + \left(\sum_{\substack{1 \leqslant \|m \|_{\infty} \leqslant H \\ \langle m , \gamma \rangle \neq 0}} \frac{|\widehat{\nu}_X(m)|^2}{\|m \|_2^2}\right)^{\frac12},
$$
and using the upper bound \eqref{eq: upper bound inverse lin comb}, we obtain the conclusion.
\end{proof}

\begin{remark}
(1) The index $c$ in the notation $\Gamma_c$ stands for \enquote{continuous}. Indeed, there is also a discrete version of Theorem \ref{Q-KW}, which we will use in the sequel \cite{BHU-II} to this paper. 

(2) In a recent work \cite[Th. 1]{borda_cuenin}, Borda and Cuenin extended the Bobkov--Ledoux inequality to $\wass_p$ for any $p \geqslant 1$. However, for $p >1$, their inequality requires that at least one of the measures is greater than $c \mathrm{Vol}$ for some positive constant $c$. Therefore, it only gives a quantitative Kronecker--Weyl theorem when the $\gamma_j$ are linearly independent over $\Qq$. Of course, since we will eventually make this assumption in the context of prime number races, this is not really a limitation. Nevertheless, there is another reason why we did not carry on with a general $p \geqslant 1$. In fact, we rely on the Kantorovich--Rubinstein duality theorem (Theorem \ref{th: kantorovich-rubinstein}) in several parts of this paper, and when $p$ is greater than $1$, there is no analogue of this theorem. A partial duality principle is stated in \cite[Lemma 13]{borda_cuenin}, but it only gives an inequality between $\wass_p$ and some dual Sobolev norms, and it seems that there is no hope to obtain an inequality in the other direction in dimension greater than $1$ (see e.g. \cite[Remark 2]{graham}).
\end{remark}

\section{An effective linear independence conjecture for Dirichlet $L$-functions}
\label{sec: lin-indep}

The linear independence conjecture (LI) that we mentioned in the introduction dates back to the 1930s. It appears in the work \cite{Wintner} of Wintner, who shed light on the relationship between this conjecture (applied to the imaginary parts of the zeros of $\zeta$) and the distribution of the error term in the Prime Number Theorem. As such, it is related to the summatory function $\psi(x)$ of the von Mangoldt function $\Lambda(n)$. The LI conjecture also appears in works related to the summatory function $M(x)$ of the Möbius function $\mu(n)$. In particular, Ingham proved in \cite{Ingham42} that LI implies $\liminf_{x \to \infty} \frac{M(x)}{\sqrt{x}} = - \infty \text{ and } \limsup_{x \to \infty} \frac{M(x)}{\sqrt{x}} =  \infty$, a statement that is much stronger than Odlyzko and te Riele's disproof of Mertens' conjecture \cite{disproof_mertens}, which was based on effective lower bounds of linear combinations of small zeros of $\zeta$. However, proving the LI conjecture seems completely out of reach, and the best result in this direction, due to Li and Radziwi\l\l, states that the proportion of zeros of $\zeta$ lying in a given vertical arithmetic progression is less than $\frac{2}{3}$ in large intervals, see \cite[Theorem 4]{radziwill_li}.

The LI conjecture can be formulated similarly for the imaginary parts of the zeros of Dirichlet $L$-functions, as in the paper \cite{hooley_bdh7} by Hooley, who appears to be the first to do so. It states that, for $q\ge 3$, the multi-set of the positive imaginary parts of the zeros of Dirichlet $L$-functions associated with characters modulo $q$ is linearly independent over $\Qq$. This version of LI attracted a lot of interest since the article \cite{RS94} of Rubinstein and Sarnak on Chebyshev's bias in prime number races. 

Since the 1980's, there has been interest in \emph{effective} linear independence statements, meaning statements that quantify how far from zero the non trivial linear combinations of imaginary parts of the zeros of $\zeta$ are. A conjecture in this direction is attributed to Monach and Montgomery \cite[p. 483]{MV}. It seems to us that the best available written reference is Lamzouri's recent article \cite[Conj. 1.1.]{lamzouri_eli}, where an effective form of the LI conjecture is stated, and a heuristic argument supporting it is written in full detail. This conjecture was also used by Ng in \cite{Ng-ELI} to obtain very general $\Omega_{\pm}$ results.
Inspired by this, we now formulate the following conjecture about imaginary parts of zeros of Dirichlet $L$-functions.

\begin{conjecture}[$\mathrm{ELI}_A(\mathcal{X}_q)$] \label{conj: ELI_A}
 Let $q \in \Zz_{\geq 1}$, $A>1$, and let $\mathcal X_q$ be a set of Dirichlet characters modulo $q$. We let $(\gamma_n)_{n \geqslant 1}$ be an enumeration in non-decreasing order of the positive imaginary parts of zeros, counted with multiplicities, of Dirichlet $L$-functions associated with the characters $\chi \in \mathcal X_q$. For all $T$ sufficiently large, one has $$\left|\sum_{j=1}^{N(T)} m_j \gamma_j\right| \gg_A N(T)^{-N(T)^A}$$ for any integers $m_1, \dots, m_{N(T)}$, not all zero, satisfying $|m_j| \leq N(T)$, where $N(T)$ is the number of such imaginary parts up to height $T$, \emph{i.e.} $N(T) = \sum_{\gamma_n \leq T} 1$.
\end{conjecture}

This conjecture extends Lamzouri's formulation to Dirichlet $L$-functions. Furthermore, when applied to the zeros of the Riemann zeta function, it is actually weaker than Lamzouri's hypothesis, which essentially amounts to ELI$_A$ for \textit{every} $A > 1$. 

\begin{remark}
\begin{enumerate}
    \item We note that if the $\gamma_j$ were to be replaced by logarithms of linearly independent algebraic numbers, the heuristic behind this conjecture is the one leading to the Lang-Waldschmidt conjecture (see \cite[p.212]{lang}), based on the pigeonhole principle. The best unconditional lower bound in this case is given by Baker's celebrated theorem on linear forms in logarithms, which will be used in the context of function fields in a follow-up to this paper.\\
    
    \item The form of our conjecture is reminiscent of the following. When $\xi$ is a transcendental number, it is possible to provide a notion of measure of transcendency of $\xi$ when there is a lower bound of the form $|P(\xi)| \gg H^{-f(n)}$ for any non-zero $P \in \Zz[X]$ with coefficients bounded in absolute value by $H$ and of degree at most $n$, and $f$ is some positive function. The value of $f$ at $n=1$ is related to the classical irrationality measure of an irrational number. More generally, this leads to the Mahler classification of transcendental numbers as $S, T$ and $U$-numbers and to the Wirsing conjecture in diophantine theory, see \cite{Bugeaud} for more about those kinds of exponents. When $P$ is of degree $n$, this amounts to a lower bound on integral linear combinations of $1, \xi, \dots, \xi^n$, with coefficients bounded in absolute value by $H$. In our case, we consider linear combinations of the positive imaginary parts of non-trivial zeros of Dirichlet $L$-functions instead of powers of $\xi$ and specialize to $H=n+1$.\\
    
    \item Another variant of the linear independence hypothesis, called QLI (for Quantitative Linear Independence), also appeared recently in \cite{SKL} to obtain joint distribution of primes in short intervals.
\end{enumerate}
\end{remark}

\section{Prime number races}
\label{sec: prime-races}
Let $q\ge 3$ and $t_q \colon (\Zz/q\Zz)^\times\to \Rr$ be a non-zero map satisfying $\langle t_q , \chi_0\rangle=0$, where $\langle \cdot , \cdot \rangle$ is the usual inner product on the space of complex valued maps defined on the group $(\Zz/q\Zz)^\times$, defined by $$\langle f, g \rangle := \frac{1}{\varphi(q)} \sum_{a \in (\Zz/q\Zz)^{\times}} f(a) \overline{g(a)},$$ and $\chi_0$ is the trivial character of this group. In the applications to prime number races we have in mind, $t_q$ will be a (normalized) difference of two indicator functions, but we keep it general for the moment. Recall that the Fourier transform of $t_q$ is defined as follows: 
\[\begin{array}{ccccc}
  \widehat{t_q} & : &  \widehat{(\Zz / q \Zz)^{\times}}  & \to & \Cc  \\
     & & \chi & \mapsto & \langle t_q, \chi \rangle,
\end{array}\] so that we have the Fourier inversion formula $t_q = \sum_{\chi} \hat{t_q}(\chi) \chi$. An important quantity in our estimates is the $L^1$ norm of $\widehat{t_q}$:  \begin{equation}\label{Littlewood-norm} \lambda(t_q):=\sum_{\chi} \bigl| \langle t_q,\chi \rangle \bigr|\, .\end{equation}
Moreover, we will denote by $\mathrm{ELI}_A(t_q)$ the conjecture $\mathrm{ELI}_A(\mathcal{X}_q)$ (Conjecture \ref{conj: ELI_A}) applied to the set of characters $\mathcal{X}_q = \mathrm{supp}(\widehat{t_q})$.

Define the prime counting functions \begin{align*}
    \pi(x;t_q):=\sum_{\substack{p\le x\\ (p,q)=1}}t_q(p)\, ,\ \ 
    \theta(x;t_q):=\sum_{\substack{p\le x\\ (p,q)=1}} t_q(p)\log p\, ,\ \ 
    \psi(x;t_q):=\sum_{\substack{n\le x\\ (n,q)=1}} t_q(n) \Lambda(n)\, ,
\end{align*}
where $\Lambda$ is the von Mangoldt function, and $\mathcal{P}(t_q):=\{\, x\ge 2\, :\, \pi(x;t_q) >0\} \, .$
Following the work of Rubinstein and Sarnak~\cite{RS94}, one can prove that, under the Generalized Riemann Hypothesis ($\mathrm{GRH}$) and the Linear Independence hypothesis (LI), the logarithmic density of $\mathcal{P}(t_q)$, \textit{i.e.} the quantity
\begin{equation}
\delta_{t_q} := \lim_{X \to \infty} \frac{1}{X} \int_{\log 2}^X\mathds{1}_{(0,\infty)}\bigl(\pi\bigl(e^y;t_q \bigr)\bigr)  \mathrm{d} y,
\end{equation}
exists and is strictly between $0$ and $1$. The objective of this section is to study the rate of convergence to this limit and deduce consequences on generalized Skewes' numbers. Specifically, we aim to establish an explicit upper bound for
\begin{equation}
\left| \delta_{t_q} - \frac{1}{X-\log 2} \int_{\log 2}^{X} \mathds{1}_{(0,\infty)}\bigl(\pi\bigl(e^y;t_q \bigr)\bigr) \mathrm{d}y \right|.
\end{equation}
For our applications, it is essential to determine how this bound depends on the modulus $q$. Since the analysis in~\cite{RS94} keeps the dependence on $q$ implicit, their results cannot be used directly in our setting. We therefore revisit their approach and adapt it to provide bounds with explicit dependence on $q$ and the function $t_q$ throughout the argument.\\
An important particular case is the race between quadratic residues and nonresidues. In what follows, let us denote \[ \ell_q:= (\rho(q)-1)\mathds{1}_{R_q}-\mathds{1}_{NR_q},\]
where $\rho(q)$ denotes the number of square roots of $1$ modulo $q$. We note that for all $x\ge2$, $\pi(x;\ell_q)>0$ if and only if \[\frac{1}{|R_q|} \sum_{a\in R_q} \pi(x;q,a) > \frac{1}{|NR_q|} \sum_{b\in NR_q} \pi(x;q,b) .\]

\subsection{The Riemann-von Mangoldt formula and applications}
In this subsection we recall some basic facts about zeros of Dirichlet $L$-functions. The Riemann-von Mangoldt formula states that if $\chi$ is a Dirichlet character modulo $q$, then the number $N(T,\chi):=\# \{\rho_\chi\, :\, 0<\Re(\rho_\chi)<1\text{ and } 0 <\Im(\rho_\chi) \le T;\ L(\rho_\chi,\chi)=0 \}$, where each zero is counted with multiplicity, satisfies: \begin{equation}\label{R-vM1}
    N(T,\chi)=\frac{T}{2\pi}\log \left(\frac{q_\chi T}{2\pi e}\right)+O(\log q_\chi T),
\end{equation} 
where $q_\chi$ is the conductor of $\chi$. Define
\begin{equation*}
    N(T,t_q) := \# \left\{ 0 < \gamma \leq T\, :\,   \begin{aligned}
  & \gamma \text{ is the imaginary part of a non-trivial zero of a Dirichlet} \\ & L\text{-function} \text{ associated with a character $\chi \in \supp(\widehat{t_q})$} 
  \end{aligned}  \right\}
\end{equation*}
where each zero is counted with its multiplicities for each corresponding $L$-function. We obtain \[N(T, t_q)=\sum_{\chi \in \supp(\widehat{t_q})} N(T,\chi).\] 

Therefore, if we define $k(t_q)$ as the unique real number satisfying
\begin{equation} \label{eq: def k(t_q)}
 \log k(t_q) = \frac{1}{|\supp(\widehat{t_q})|} \sum_{\chi \in \supp(\widehat{t_q})} \log q_{\chi},
\end{equation}
we have the following result.
\begin{lemma}\label{lem: N(T,t_q)-expression} Let $q \geq 3$. Then for $T\geqslant 1$:
\begin{equation}\label{eq: N(T,t_q) estimate}
N(T, t_q)=\frac{|\supp(\widehat{t_q})|T}{2\pi}\log\left(\frac{k(t_q)T}{2\pi e}\right)+O(|\supp(\widehat{t_q})| \log (k(t_q)T)).
\end{equation}
\end{lemma}
We now apply~\eqref{R-vM1} to estimate sums over zeros.
\begin{corollary}\label{sums-zeros}
Assume $\mathrm{GRH}$ for Dirichlet $L$-functions modulo $q$. We have the following:
\begin{enumth}
        \item \begin{equation*}
    S_1(t_q):= \left| \sum_{\chi \in \supp(\hat{t_q})} \langle t_q,\chi\rangle \sum_{\gamma_\chi} \frac{1}{(\tfrac{1}{2}+i\gamma_\chi)(\tfrac{3}{2}+i\gamma_\chi)}\right| \ll \lambda(t_q) \log q
\end{equation*}
\item \begin{equation*}
    S_2(t_q):= \left|\sum_{\chi \in \supp(\hat{t_q})} \langle t_q,\chi\rangle \sum_{\gamma_\chi} \frac{1}{\tfrac{1}{4}+\gamma_\chi ^2}\right|\ll \lambda(t_q) \log q.
    \end{equation*}
\item  Let $\chi,\lambda$ be Dirichlet characters modulo $q$. Then, uniformly for $T,Y>3$ 
\begin{equation*}
 \sum_{|\gamma_\chi|, |\gamma_\lambda|\ge T}\frac{1}{|\gamma_\chi \gamma_\lambda|} \min\left(Y,\frac{1}{|\gamma_\chi-\gamma_\lambda|}\right) \ll (\log q)^2 \left(Y\frac{(\log T)^2}{T}+\frac{(\log T)^3}{T}\right).
\end{equation*}
\end{enumth} 
\end{corollary}


\begin{proof}
    We only prove the second bound as the first follows from the same argument.
    Define, for $x\in \Rr$, \[h(x)=\frac{1}{\tfrac{1}{4}+x^2}.\] A summation by parts yields \[  \sum_{\gamma_{\chi}}h(\gamma_\chi)= - \int_0^\infty N(x,\chi)h'(x) \mathrm{d}x. \]
    Using the fact that $|h'(x) | \ll x^{-3}$ and the estimate \eqref{R-vM1} for $N(x,\chi)$, we deduce that $\sum_{\gamma_\chi}h(\gamma_\chi)\ll\log q$, and the result follows by applying the triangle inequality.
    A detailed proof of the third estimate is given in the appendix starting on page \pageref{appendix_proof_sum_zeros}.
\end{proof}

\subsection{Consequences of the explicit formula with explicit bounds on $q$} \label{subsec: consequences explicit in q}
In this subsection, we revisit the proofs of Rubinstein--Sarnak~\cite{RS94} and make the dependencies on $q$ explicit. Let $q\ge 3$ and $t_q \colon (\Zz/q\Zz)^\times \to \Rr$ be a non-zero function satisfying $\langle t_q,\chi_0\rangle =0$. Define
\[ \begin{array}{ccccl}
    r_q& :& (\Zz/q\Zz)^\times &\longrightarrow &\Rr\\
    && a &\longmapsto& \# \{x\in (\Zz/q\Zz)^\times\, :\, x^2=a\}\, ,
\end{array} \]
and \[\begin{array}{ccccc}
    U_q & : & (\Zz/q\Zz)^\times &\longrightarrow &(\Zz/q\Zz)^\times\\
   && x&\longmapsto & x^2 
\end{array} \]

Recall the estimates of the prime number theorem under GRH: 
\begin{lemma}[\cite{MV}*{Theorem 13.7}] \label{lemma:TNP}
Assume $\mathrm{GRH}$ for Dirichlet $L$-functions modulo $q$. Then for all Dirichlet characters $\chi$ modulo $q$, we have uniformly for $x\ge 2$,
\[ \psi(x;\chi)=\delta_{\chi=\chi_0}x+O\bigl( \sqrt{x} (\log x) (\log qx) \bigr) \, ,\]
and 
\[ 
 \theta(x;\chi)=\delta_{\chi=\chi_0}x+O\bigl( \sqrt{x} (\log x) (\log qx) \bigr) \, ,
\]
where $\delta_{\chi=\chi_0}$ is equal to $1$ if $\chi=\chi_0$ and $0$ otherwise.
\end{lemma}

We now state the explicit formula with explicit bounds in terms of $q$: 
\begin{lemma}\label{explicit-formula}
Assume $\mathrm{GRH}$ for Dirichlet $L$-functions modulo $q$. Then, uniformly for $x\ge 2$ and $T>0$, we have 
\[ \psi(x;t_q)=-\sqrt{x}\sum_{\chi \ne \chi_0}\langle t_q,\chi\rangle \sum_{|\gamma_\chi| \le T}\frac{x^{i\gamma_\chi}}{\tfrac{1}{2}+i\gamma_\chi}+O\left( \lambda(t_q) (\log q)^2\left(\log x + \frac{x}{T}(\log xT)^2 \right)\right) ,\]
where $\lambda(t_q)$ is defined in~\eqref{Littlewood-norm}.
\end{lemma}
\begin{proof}
    It suffices to write \[ \psi(x;t_q)=\sum_{\chi \ne \chi_0}\langle t_q,\chi \rangle \psi(x;\chi)\,\]
    and to use the standard explicit formula \cite{MV}*{Theorem 12.12} for each $\psi(x;\chi)$.
\end{proof}

In order to simplify notations, we will denote $t_q^*:=t_q\circ U_q \colon a \mapsto t_q(a^2)$ so that \[\langle t_q^*,\chi_0\rangle =\langle t_q , r_q \rangle,\] and to keep track of future dependencies in $q$, we define \begin{equation} \label{eq:Ctq}
    C(t_q) = \max(\lambda(t_q)(\log q)^2, \lambda(t_q^*) \log q).
\end{equation}
We note that \[ |\langle t_q,r_q\rangle|\leq \|t_q^*\|_\infty\leq\|t_q\|_\infty \leq \lambda(t_q) .\]
\begin{lemma}\label{relating-pi-psi}
Assume $\mathrm{GRH}$. Then, uniformly for $x\ge 2$, we have
\[ \pi(x;t_q)=-\frac{\sqrt{x}}{\log x} \langle t_q,r_q \rangle+\frac{\psi(x;t_q)}{\log x}+O\left(\frac{C(t_q) x^{1/2}}{(\log x)^2}\right).\]

\end{lemma}
\begin{proof}
First, we separate the contributions of primes, squares of primes, and higher powers of primes, to get $$\psi(x; t_q) = \theta(x;t_q) + \theta(x^{1/2}; t_q^*) + O(x^{1/3} \|t_q\|_{\infty}),$$ where we used Chebyshev's bound $\theta(x^{1/k}) \ll x^{1/k}$ for the remainder term. Lemma \ref{lemma:TNP} and the decomposition $t_q^* = \sum_{\chi} \langle t_q^*, \chi \rangle \chi$ yield \begin{equation} \label{theta}\theta(x; t_q) = \psi(x; t_q) - \langle t_q, r_q \rangle \sqrt{x} + O(\lambda(t_q^*)x^{1/4} \log x \log(qx) + x^{1/3} \|t_q\|_{\infty}).\end{equation} The $O$ term clearly is $O(C(t_q)x^{1/3})$.
A summation by parts gives \[ \pi(x;t_q)=\frac{\theta(x;t_q)}{\log x}+\int_2^x\frac{\theta(u;t_q)}{u(\log u)^2}\mathrm{d}u \, .\]
By (\ref{theta}), we have $$\frac{\theta(x;t_q)}{\log x} = -\frac{\sqrt{x}}{\log x} \langle t_q,r_q \rangle+\frac{\psi(x;t_q)}{\log x}+ O\left(\frac{C(t_q)x^{1/3}}{\log x}\right)$$ and $$\int_2^x\frac{\theta(u;t_q)}{u(\log u)^2}\mathrm{d}u = \int_2^x\frac{\psi(u;t_q)}{u(\log u)^2}\mathrm{d}u + O\left(\frac{C(t_q) x^{1/2}}{(\log x)^2}\right).$$
Therefore, it suffices to prove that \[ \int_2^x \frac{\psi(u;t_q)}{u(\log u)^2}\mathrm{d} u \ll \frac{C(t_q) x^{1/2}}{(\log x)^2}.\]
Define $G(x;t_q):=\int_2^x\psi(u;t_q) \mathrm{d}u$. We integrate in Lemma~\ref{explicit-formula}, and let $T \to \infty$ to obtain \[G(x;t_q)=-\sum_{\chi}\langle t_q,\chi \rangle \sum_{\gamma_\chi}\frac{x^{3/2 + i\gamma_\chi}}{(\tfrac{1}{2}+i\gamma_\chi)(\tfrac{3}{2}+i\gamma_\chi)}+O\bigl(\lambda(t_q)(\log q)^2 x\log x\bigr),\] where the series converges absolutely by Corollary ~\ref{sums-zeros} i). That same bound implies in particular that $G(x;t_q)\ll C(t_q) x^{3/2}$. Integrating by parts, we obtain
    \begin{align*}
        \int_2^x \frac{\psi(u;t_q)}{u (\log u)^2}\mathrm{d}u&=\frac{G(x;t)}{x(\log x)^2}+\int_2^x\frac{(\log u)^2+2\log u }{u^2 (\log u)^4}G(u;t_q) \mathrm{d}u\\
        &\ll C(t_q) \frac{x^{1/2}}{(\log x)^2}\, .
    \end{align*}
\end{proof}

Define for all $y\ge 1$ \begin{equation}\label{Def-E(y)}
E(y)=E_{t_q}(y):=\frac{y}{e^{y/2}}\pi\bigl( e^y;t_q\bigr)    
\end{equation}  and for $T \geq 1$, 
\begin{equation}\label{Def-ET(y)}  E^{(T)}(y)=E_{t_q}^{(T)}(y):=-\langle t_q, r_q \rangle-\sum_\chi \sum_{|\gamma_\chi| \le T}\langle t_q ,\chi \rangle \frac{e^{iy\gamma_\chi}}{\tfrac{1}{2}+i\gamma_\chi}.\end{equation}


\begin{lemma}\label{almost-periodicity} Assume $\mathrm{GRH}$. Then, uniformly for $Y\ge \log T$, we have:
\[\frac{1}{Y}\int_{\log 2}^Y \bigl|E(y)-E^{(T)}(y)\bigr|\mathrm{d}y \ll C(t_q) \left(\frac{\log T}{\sqrt{T}}+\frac{1}{\sqrt Y}\right).\]
\end{lemma}
\begin{proof}
Combining Lemmas~\ref{explicit-formula} and~\ref{relating-pi-psi}, we obtain, for $Y \geq \log T$, $$E(y) - E^{(T)}(y) = -\sum_{\chi}\sum_{T<|\gamma_\chi|\le e^Y}\langle t_q,\chi\rangle \frac{e^{iy\gamma_\chi}}{\tfrac{1}{2}+i\gamma_\chi}+O\left(C(t_q) \left(\frac{1}{y}+\frac{e^{y/2} Y^2}{e^Y}\right)\right).$$ Taking squares and integrating between $\log 2$ and $Y$, we obtain \begin{align*}& \int_{\log 2} ^Y \bigl |E(y)-E^{(T)}(y) \bigr|^2 \mathrm{d}y \ll \int _{\log 2}^Y \left| \sum_\chi \langle t_q, \chi \rangle \sum_{T<  |\gamma_\chi|\le e^Y}\frac{e^{iy \gamma_\chi}}{\tfrac{1}{2}+i\gamma_\chi}\right|^2 \mathrm{d}y+O(C(t_q)^2 ) \\
    &=\int_{\log 2}^Y \sum_{\chi, \lambda} \langle t_q, \chi\rangle \overline{\langle \lambda, t_q \rangle}\sum_{T< |\gamma_\chi|, |\gamma_\lambda| \le e^Y}\frac{e^{iy(\gamma_\chi-\gamma_\lambda)}\mathrm{d}y}{(\tfrac{1}{2}+i\gamma_\chi)(\tfrac{1}{2}-i\gamma_\lambda)}+O( C(t_q)^2)\\
    &\ll \sum_{\chi,\lambda \ne \chi_0} \langle \chi,t_q\rangle \overline{\langle \lambda, t_q \rangle} \sum_{|\gamma_\chi|, |\gamma_\lambda|> T}\frac{1}{|\gamma_\chi \gamma_\lambda|}\min\left(Y,\frac{1}{|\gamma_\chi-\gamma_\lambda|}\right)+O( C(t_q)^2) \\
    &\ll \lambda(t_q)^2 (\log q)^2 \left(Y\frac{(\log T)^2}{T}+\frac{(\log T)^3}{T}\right)+O( C(t_q)^2)\\
    &\ll C(t_q)^2\left(1+Y\frac{(\log T)^2}{T}\right)
    \end{align*}
    where we used Corollary~\ref{sums-zeros} iii) to deduce the second-to-last estimate. Applying the Cauchy-Schwarz inequality, we deduce that \begin{align*}\frac{1}{Y}\int_{\log 2}^Y \bigl|E(y)-E^{(T)}(y)\bigr|\mathrm{d}y &\le \left( \frac{1}{Y} \int_{\log 2}^Y \bigl | E(y) - E^{(T)}(y) \bigr|^2 \mathrm{d}y \right)^{1/2} \\ &\ll C(t_q) \left(\frac{\log T}{\sqrt{T}}+ \frac{1}{\sqrt Y}\right). \end{align*}
    This finishes the proof of the Lemma.

\end{proof}

\begin{remark}
Lemma \ref{almost-periodicity} is the analog of \cite[Lemma 2.2]{RS94}, with an explicit dependency in $q$. Note however that we have a $\frac{1}{\sqrt Y}$ term instead of $\frac{\log T}{\sqrt{TY}}$ in our statement. This is because there is a small mistake in \cite[Lemma 2.2]{RS94}, as the authors drop a $O(1)$ term in their proof, which cannot be done unless $Y$ is larger than $\frac{T}{\log^2T}$. This does not change the fact that, for fixed $q$, the above quantity goes to zero as $T \to \infty$ since $Y$ is restricted to being larger than $\log T$.
\end{remark}

\subsection{Existence of a limiting distribution and quantitative convergence}

Let us define a probability measure $\mu_X$ on $\Rr$ to be the pushforward of the normalized Lebesgue measure on $[\log 2\, , X]$ by $E=E_{t_q}$, so that for all bounded and continuous functions $f$ we have \begin{equation} \label{eq: f de E}  \int_\Rr f(y) \mathrm{d}\mu_X(y)=\frac{1}{X-\log 2}\int_{\log 2}^X f(E(y))\mathrm{d}y. \end{equation}
Similarly, we define $\mu_X^{(T)}$ to be the unique probability measure on $\Rr$ such that for all continuous and bounded functions $f$, 
\begin{equation} \label{eq: f de ET} \int_{\Rr}f(y)\mathrm{d}\mu_X^{(T)}(y)=\frac{1}{X-\log 2}\int_{\log 2}^X f\left(E^{(T)}(y)\right)\mathrm{d}y .\end{equation}

As a first step in this subsection, we want to control the distance $\wass_1(\mu_X^{(T)}, \mu_X)$ (where $\wass_1$ is defined with respect to the usual Euclidean metric on $\Rr$). In view of \eqref{eq: f de E} and \eqref{eq: f de ET}, this should follow from estimates for the size of the error term when approximating $E$ by $E^{(T)}$, which is precisely the content of \S \ref{subsec: consequences explicit in q}. This is what we do in the following lemma. 

\begin{lemma} \label{lem: trunc step W_1}
Assume $\mathrm{GRH}$. For all $T\ge 3$ and all $X\ge \log T$ we have 
\[\wass_1\left(\mu_X^{(T)},\mu_X\right) \ll C(t_q) \left(\frac{\log T}{\sqrt{T}}+\frac{1}{\sqrt X}\right).\]
\end{lemma}

\begin{proof}
    Let $u:\Rr\to\Rr$ be a $1$-Lipschitz bounded function. By Lemma~\ref{almost-periodicity} we have
    \begin{align*}
        \left|\int_\Rr u \mathrm{d}\mu_X^{(T)}-\int_\Rr u \mathrm{d}\mu_X \right|&\le \frac{1}{X-\log 2}\int_{\log 2}^X \bigl |E(y)-E^{(T)}(y)\bigr|\mathrm{d}y\\ &\ll C(t_q) \left(\frac{\log T}{\sqrt{T}}+\frac{1}{\sqrt X}\right).
    \end{align*} 
    The conclusion follows by taking the supremum over $u$, thanks to Theorem \ref{th: kantorovich-rubinstein}.
\end{proof}

\begin{remark}
The restriction $X \geq \log T$ is natural for our purpose, as the remainder term in the explicit formula in Lemma \ref{explicit-formula} does not go to zero as $T \to \infty$ and $x$ is fixed, owing to the contribution of trivial zeros, a potential constant term and jump discontinuities of size roughly $\log x$ at prime power values of $x$. In particular, for fixed $X$, $\mu_X^{(T)}$ does not converge weakly to $\mu_X$ as $T \to \infty$.
\end{remark}

Another fruitful point of view on $\mu_X^{(T)}$ is that it can be seen as a pushforward measure of a measure $\nu_X^{(T)}$ on $\bigl(\Ss^1\bigr)^{N(T, t_q)}$ which is defined as in the statement of the Kronecker--Weyl theorem. Let us make this claim more precise.

Let $(\gamma_n)_{n \geq 1}$ be the enumeration in non-decreasing order of the positive imaginary parts of the non trivial zeros, counted with multiplicities, of Dirichlet $L$-functions associated with characters $\chi \in \mathrm{supp}(\widehat{t_q})$. By taking the first $N(T, t_q)$ zeros, we can define a measure $\nu_X^{(T)}$ on $\bigl(\Ss^1\bigr)^{N(T, t_q)}$ as in Theorem \ref{Q-KW} with the specific choice $x_0 = \log 2$, meaning that for all continuous maps $h \colon (\Ss^1)^{N(T, t_q)} \to \Cc$,
\begin{equation} \label{eq: def nuXT}
    \int_{\bigl(\Ss^1\bigr)^{N(T, t_q)}}h\mathrm{d}\nu_X^{(T)}=\frac{1}{X-\log 2}\int_{\log 2}^X h\bigl(e^{iy\gamma_1},\dots,e^{iy\gamma_{N(T, t_q)}}\bigr)\mathrm{d}y.
\end{equation}

Since $t_q$ is real-valued, using the symmetries of the zeros, we can write 
\[E^{(T)}(y)=-\langle t_q,r_q \rangle -2\Re\left(\sum_{n=1}^{N(T, t_q)}b_n e^{i \gamma_n y}\right). \] where $b_n = \frac{\langle t_q, \chi_n \rangle}{\frac{1}{2} + i \gamma_n}$ and $\chi_n$ is the character corresponding to the zero $\frac{1}{2} + i \gamma_n$.
Therefore, if we define \begin{align} \label{def gT}
    g^{(T)}\, :\, \left(\Ss^1\right)^{N(T, t_q)}&\longrightarrow \Rr\\
    (z_1,\dots,z_{N(T, t_q)})&\longmapsto -\langle t_q,r_q \rangle-2\Re\left(\sum_{n=1}^{N(T, t_q)}b_n z_n\right), \nonumber
\end{align}
it is now a formal consequence of the definitions of the measures involved that $\mu_X^{(T)}$ is the pushforward measure of $\nu_X^{(T)}$ by $g^{(T)}$, meaning that for all continuous functions $f$,
\[
\int_{\Rr}f(y)\mathrm{d}\mu_X^{(T)}(y) =\int_{(\Ss^1)^{N(T, t_q)}}f\circ g^{(T)}\mathrm{d}\nu_X^{(T)}.
\]

At this point, previous works on the subject have shown that, for fixed $T$ and as $X$ goes to infinity, the measure $\nu_X^{(T)}$ converges weakly to the normalized Haar measure of a subtorus $\Gamma^{(T)}$ of $\bigl(\Ss^1\bigr)^{N(T, t_q)}$, which we denote by $\lambda_T$. This subtorus $\Gamma^{(T)}$ can be determined from the linear relations between the imaginary parts $\gamma \in (0, T]$ as in \cite{bailleul_explicit}. Then, it suffices to take the pushforward measure via $g^{(T)}$ to deduce that $\mu_X^{(T)}$ converges weakly, as $X$ goes to infinity, to the measure $\mu^{(T)}$ which satisfies
\begin{equation} \label{eq: def mu^T}
\int_\Rr f \mathrm{d} \mu^{(T)} = \int_{\Gamma^{(T)}}f\circ g^{(T)}\mathrm{d} \lambda_T.
\end{equation}
for all continuous functions $f$. Moreover, assuming the linear independence hypothesis LI, one gets $\Gamma^{(T)} = \bigl(\Ss^1\bigr)^{N(T, t_q)}$. 
Let us show that $(\mu^{(T)})_{T \geq 3}$ also converges with respect to the metric $\wass_1$.

\begin{lemma}\label{wass-almostperiodicity}
Assume $\mathrm{GRH}$. The sequence of measures $(\mu^{(T)})_{T \geqslant 3}$ converges in $(\mathscr{P}_1(\Rr), \wass_1)$ towards a probability measure $\mu$ with a finite first moment, and \begin{equation}\label{wass-mu-muT}
    \wass_1\left(\mu^{(T)},\mu\right)\ll  C(t_q) \frac{\log T}{\sqrt{T}} \cdot
\end{equation}
\end{lemma}
\begin{proof}
For all $T\ge 3$ and all $X\ge \log T$, we have 
\[\wass_1\left(\mu_X^{(T)},\mu_X\right) \ll C(t_q) \left(\frac{\log T}{\sqrt{T}}+\frac{1}{\sqrt{X}}\right)\] by Lemma \ref{lem: trunc step W_1}.
By the triangle inequality, we deduce
  \[\wass_1\left(\mu_X^{(T)},\mu_X^{(S)}\right) \ll C(t_q) \left(\frac{\log T}{\sqrt{T}}+\frac{\log S}{\sqrt{S}}+\frac{1}{\sqrt{X}}\right),\]
and letting $X \to \infty$ gives that for all $T,S \ge 3$,
\begin{equation}\label{Cauchy-wass}\wass_1\left(\mu^{(T)},\mu^{(S)}\right) \ll C(t_q) \left(\frac{\log T}{\sqrt{T}}+\frac{\log S}{\sqrt{S}}\right).\end{equation} 
 Therefore, the sequence $(\mu^{(T)})_{T \ge 3}$ is Cauchy, and since the metric space $(\mathcal{P}_1(\Rr),\wass_1)$ is complete by Proposition \ref{prop: wass properties} (3), we deduce that $(\mu^{(T)})_{T \ge 3}$ converges in law to a measure $\mu$ with a finite first moment.
  Moreover, letting $S\to \infty$ in \eqref{Cauchy-wass} we deduce \begin{equation*}
    \wass_1\left(\mu^{(T)},\mu\right)\ll  C(t_q) \frac{\log T}{\sqrt{T}} \cdot
\end{equation*}
\end{proof}

Since we want to keep track of the rates of convergence in this last step, we anticipate the use of Proposition \ref{prop: comp lip} and focus for now on the determination of a Lipschitz constant for the map $g^{(T)}$.

\begin{lemma} Assume $\mathrm{GRH}$. If we endow $\left(\Ss^1\right)^{N(T, t_q)}$ with the metric $\rho$ defined in Section $\ref{sec: kronecker-weyl}$, then the function $g^{(T)}$ is a $D_{t_q}$-Lipschitz function, where \[D_{t_q} :=2 \left(\sum_\chi \sum_{\gamma_\chi }\frac{|\langle t_q,\chi \rangle|^2}{\tfrac{1}{4}+\gamma_\chi^2}\right)^{1/2}.\]
Moreover, this Lipschitz constant satisfies
\[D_{t_q} \ll \lambda(t_q) \sqrt{\log q} \leq C(t_q).\]
    
\end{lemma}

\begin{proof}   For all $z,z' \in (\Ss^1)^{N(T, t_q)}$, we have 
  \[
        |g^{(T)}(z) - g^{(T)}(z')| = 2 \left| \Re\left( \sum_{n = 1}^{N(T, t_q)} b_n z_n \right) - \Re \left(  \sum_{n = 1}^{N(T, t_q)} b_n z'_n \right) \right| 
        \leqslant 2 \sum_{n = 1}^{N(T, t_q)} |b_n| |z_n - z'_n |. 
    \]
    Now, we use the fact that the Euclidean distance $|z- z'|$ is bounded above by the Riemannian metric $\ell(z,z')$, and the Cauchy--Schwarz inequality to obtain
    \[   |g^{(T)}(z) - g^{(T)}(z')|  \leqslant 2\left( \sum_{n = 1}^{N(T, t_q)} |b_n|^2 \right)^{1/2} \left( \sum_{n = 1}^{N(T, t_q)} \ell(z_n, z'_n)^2 \right)^{1/2} = 2\left( \sum_{n = 1}^{N(T, t_q)} |b_n|^2 \right)^{1/2} \rho(z,z'). \]
    Thus, by Corollary~\ref{sums-zeros}, 
    \[ 2 \left(\sum_{n = 1}^{N(T, t_q)} |b_n|^2\right)^{1/2} = 2\left( \sum_\chi \sum_{0<\gamma_\chi \le T} \frac{|\langle t_q,\chi \rangle|^2}{\tfrac{1}{4}+\gamma_\chi^2}\right)^{1/2} \le D_{t_q} .\] To conclude, we use the fact that $\sum_{\gamma_\chi} 1/\bigl(\frac{1}{4}+\gamma_\chi^2\bigr)\ll \log q$, which is proved in Corollary \ref{sums-zeros}, to obtain $D_{t_q} \ll (\sum_\chi |\langle t_q,\chi\rangle|^2)^{1/2}\sqrt{\log q}\leq \lambda(t_q) \sqrt{\log q} \leq C(t_q)$.
    
\end{proof}

In the following Lemma, we use our stronger assumption $\mathrm{ELI_A}(t_q)$, and replace the usual Kronecker--Weyl theorem by its quantitative form, to deduce an upper bound to $\wass_1\left(\mu_X^{(T)},\mu^{(T)}\right)$.

\begin{lemma}\label{wass-muXT-muT}
Assume $\mathrm{GRH}$ and $\mathrm{ELI_A}(t_q)$ for some $A>1$. For all $T$ sufficiently large and $X\ge N(T, t_q)^{2N(T, t_q)^A}$, we have 
\[\wass_1\left(\mu_X^{(T)},\mu^{(T)}\right)\ll_A C(t_q)N(T, t_q)^{-\tfrac{1}{2}} .\]
\end{lemma}

\begin{proof}
    Let $T>0$ be sufficiently large so that $N(T, t_q) \geq 1$ by Lemma \ref{lem: N(T,t_q)-expression}.
    Denote by $\gamma^{(T)}$ the element $(\gamma_1, \dots, \gamma_{N(T, t_q)})$ of $(\Rr_{>0})^{N(T, t_q)}$, where the $\gamma_i$'s are the positive imaginary parts of the $L(s, \chi)$ with $\chi$ being the Dirichlet characters modulo $q$ such that $\langle t_q, \chi \rangle \neq 0$. The $\mathrm{ELI}_A(t_q)$ hypothesis (Conjecture \ref{conj: ELI_A} with $\mathcal{X}_q = \mathrm{supp}(\widehat{t_q})$) implies that for all $m:=(m_j)_{1 \le j \le N(T, t_q)} \in \Zz^{N(T, t_q)}\setminus\{0\}$ with $\|m\|_\infty \le N(T, t_q)$, we have \[ \frac{1}{|\langle m, \gamma^{(T)}\rangle |} = \frac{1}{|\sum_{j = 1}^{N(T, t_q)} m_j \gamma_j |} \ll_A N(T, t_q)^{N(T, t_q)^A}.\]
    Thus
    \[  \sum_{\substack{ \|m\|_\infty \le N(T, t_q) \\ m \ne 0}}\frac{1}{\|m\|_2^{2}\cdot |\langle m, \gamma^{(T)}\rangle|^2}\ll_A N(T, t_q)^{3N(T, t_q)^A},\]
    where we bounded trivially each $1/ \|m\|_2^2$ by $1$ and used the fact that the sum contains $ O\left((2N(T, t_q))^{N(T, t_q)}\right) \ll_A N(T, t_q)^{N(T, t_q)^A}$ terms (we do not need to be more precise here since optimizing this bound would gain at most a small power of $\log \log X$ in the proof of Theorem \ref{quantitative-limdist} below). Using Theorem~\ref{Q-KW} with $H=N(T, t_q)$, we deduce that for all $X\ge N(T, t_q)^{2N(T, t_q)^A}$
    \begin{align*} \wass_1\left(\nu_X^{(T)},\lambda_T\right) &\ll_A \frac{\sqrt{N(T, t_q)}}{N(T, t_q)}+\frac{N(T,t_q)^{\frac{3}{2}N(T, t_q)^A}}{X} \\
    & \ll_A N(T, t_q)^{-\tfrac{1}{2}} ,
    \end{align*}
    where we used the inequality $$\frac{N(T,t_q)^{\frac{3}{2}N(T, t_q)^A}}{X}\le N(T,t_q)^{-\tfrac{1}{2}N(T,t_q)^A}\le N(T, t_q)^{-\tfrac{1}{2}} .$$
    Using Proposition~\ref{prop: comp lip} we deduce that \[ \wass_1\left(\mu_X^{(T)},\mu^{(T)}\right)\ll_A D_{t_q}N(T, t_q)^{-\tfrac{1}{2}} \ll_A C(t_q)N(T, t_q)^{-\tfrac{1}{2}}  .\]
    The result follows.
\end{proof}


We now prove the main theorem of this section which implies in particular that $\mu$ is in fact the limiting distribution of $E$. 

\begin{theorem}\label{quantitative-limdist}
    Assume $\mathrm{GRH}$ and $\mathrm{ELI_A}(t_q)$ for some $A>1$. Let $q\ge 3$ and let $$X\ge (|\supp(\hat{t_q})| \log k(t_q) )^{(\mathcal{L} |\supp(\hat{t_q})| \log k(t_q)) ^A}, $$
    where $\mathcal{L}>0$ is an absolute effective constant. We have 
    \[ \wass_1 \bigl(\mu,\mu_X\bigr) \ll_A C(t_q)^{3/2} (\log C(t_q))^{\tfrac{1}{2A}} (\log \log X)^2(\log X)^{-\tfrac{1}{2A}} \, .\]
\end{theorem}
\begin{proof} 
Recall that $N(T, t_q) \underset{T \to \infty}{\sim}  \frac{|\supp(\hat{t_q})|T}{2\pi} \log(k(t_q)T)$ uniformly for $T \geq 1$ by Lemma \ref{lem: N(T,t_q)-expression}. Denote by
\[
\tp := |\supp(\hat{t_q})|T \log(k(t_q)T).
\]
Let $X$ be as in the statement and let $T >0$ be such that 
\[
X = \tp^{ \tp^A}.
\]
A large enough (but absolute) choice of $\mathcal{L}$ implies that $T$ is sufficiently large to have $1 \le N(T,t_q)^A \le \tp^A/2$. Hence \[
X^{-1/2}\le  N(T,t_q)^{-1/2} \ll \frac{1}{\sqrt{T}} \ll \frac{\log T}{\sqrt{T}}\quad \text{and}\quad X\ge N(T, t_q)^{2N(T,t_q)^A}.
\]
Using the triangle inequality and Lemmas \ref{lem: trunc step W_1}, \ref{wass-almostperiodicity}, and \ref{wass-muXT-muT}  we obtain: 
    \begin{align*}
        \wass_1(\mu,\mu_X)&\le \wass_1\bigl(\mu, \mu^{(T)}\bigr)+\wass_1\bigl(\mu^{(T)},\mu_X^{(T)}\bigr)+\wass_1\bigl(\mu_X^{(T)},\mu_X\bigr) \\
        &\ll_A C(t_q) \left( \frac{\log T}{\sqrt{T}}+ N(T, t_q)^{-\tfrac{1}{2}}+X^{-\tfrac{1}{2}}\right)\\
        &\ll_A C(t_q) \frac{\log T}{\sqrt{T}}
    \end{align*} 
Moreover, we have
\begin{align*}
\log X = \tp^A \log \tp &\ll C(t_q)^A T^A (\log T)^{A} \log(C(t_q) T) \\ &\ll C(t_q)^A T^A (\log C(t_q)) (\log T)^{2A},
\end{align*}
from which we deduce the following upper bound:
\[\frac{1}{\sqrt{T} \log T} \ll \sqrt{C(t_q)} (\log C(t_q))^{\tfrac{1}{2A}}(\log X)^{-\tfrac{1}{2A}}.\]
Thus,
\begin{align*}
    \frac{\log T}{\sqrt{T}} \ll (\log T)^2 \sqrt{C(t_q)} (\log C(t_q))^{\tfrac{1}{2A}} (\log X)^{-\tfrac{1}{2A}}.
\end{align*}
Finally,
\[
\log T \leqslant \log\left( |\mathrm{supp}(\widehat{t_q})| T \right) \ll \log \tp \ll_A \log \log X,
\]
hence the result.
\end{proof}

\subsection{The logarithmic density of Chebyshev's bias} \label{sec:Lipschitz}
Following the work of Rubinstein and Sarnak, and assuming $\mathrm{GRH}$ and LI, the Fourier transform of $\mu$ can be expressed as: 
\[ \widehat{\mu}(\xi)=\exp\left(-i\langle t_q,r_q \rangle\xi\right) \prod_{\chi \ne \chi_0} \prod_{\gamma_\chi>0} \J\left(\frac{2|\langle t_q,\chi\rangle \xi|}{\sqrt{\tfrac{1}{4}+\gamma_\chi^2}}\right),\]
where $\J$ is the Bessel function of the first kind. Using this expression, we see that $\widehat{\mu}$ decays rapidly at $\infty$ which shows that $\mu$ is absolutely continuous with a real-analytic density that we will denote $f_{t_q}$. This proves that 
\[ \delta=\delta_{t_q}= \lim_{X \to \infty} \frac{1}{X-\log 2}\int_{\log 2}^X \mathds{1}_{(0,\infty)}(E(y)) \mathrm{d}y=\mu(0,\infty) \]
exists and satisfies $0<\delta<1$. We now estimate the rate of this convergence. 
Let $L>0$, define $h_L^{\pm}:\Rr\to \Rr$ for $x\in \Rr$ by 
\[h_L^+(x)=\begin{cases}
0\qquad \qquad\  \text{ if }x\le -1/L\\
L(x+1/L) \text{ if } -1/L \le x \le 0\\
1\qquad \qquad\    \text{ if } x\ge 0
\end{cases}
\text{ and } h_L^-(x)=\begin{cases}
0\ \ \text{ if }x\le 0\\
Lx \text{ if } 0 \le x \le 1/L\\
1\ \  \text{ if } x\ge 1/L
\end{cases}. \]
\begin{lemma}\label{Lipschitz-approximation}
Let $\nu$ be a probability measure on $\Rr$ with density $f$, and assume that $f\in L^\infty (\Rr)$. We have for all $L>0$
\[ \left|\nu(0,\infty)-\int_\Rr h_L^{\pm}(x) \mathrm{d}\nu(x)\right| \le \frac{\|f\|_\infty}{2L}\, . \]
\end{lemma}
\begin{proof}
    We prove the Lemma for $h_L^+$, the same argument applies for $h_L^-$. Let $L>0$ we have 
    \begin{align*} \left| \nu(0,\infty)-\int_\Rr h_L^+(x) \mathrm{d}\nu(x)\right|&=\left | \int_{-1/L}^0 L(x+1/L)f(x) \mathrm{d}x\right|\\
    &\le \|f\|_\infty \left[ \frac{L}{2}(x+1/L)^2 \right]_{-1/L}^0=\frac{\|f\|_\infty}{2L} .
    \end{align*}
\end{proof}
\begin{lemma} \label{density-wass}
Let $\nu,\nu'\in \mathcal{P}_1(\Rr)$ and assume that $\nu$ has a density $f \in L^{\infty}(\Rr)$. Then we have 
\[  |\nu(0,\infty)-\nu'(0,\infty)|\le 2\left( \|f\|_\infty \wass_1(\nu,\nu') \right)^{1/2}. \]
\end{lemma}
\begin{proof} By definition of $h_L^{+}$ we have
 \[\nu'(0,\infty)\le \int_\Rr h_L^+ \mathrm{d}\nu'.\]
    Moreover, since $h_L^{+}$ is $L$-Lipschitz, Theorem \ref{th: kantorovich-rubinstein} implies that
    \[ \int_\Rr h_L^+ \mathrm{d}\nu' - \int_\Rr h_L^+\mathrm{d}\nu \le L \wass_1(\nu,\nu') .\]
    Hence \[ \nu'(0, \infty) \le L \wass_1(\nu, \nu') + \int_\Rr h_L^+\mathrm{d}\nu.\]
    Thanks to Lemma~\ref{Lipschitz-approximation} this implies the following inequality~: 
    \[\nu'(0,\infty)-\nu(0,\infty)\le L\wass_1(\nu,\nu')+\frac{\|f\|_\infty}{2L}.\]
    Similarly, using $h_L^-$, we deduce that 
    \[ \nu(0,\infty)-\nu'(0,\infty) \le L \wass_1(\nu,\nu')+\frac{\|f\|_\infty}{2L}\, .\]
    The result now follows from taking $L=(\|f\|_\infty /\wass_1(\nu,\nu'))^{1/2}$.
\end{proof}

In the case of our limiting distribution $\mu$, since $\widehat{\mu}\in L^1(\Rr)$, by the Fourier inversion formula we deduce that $f_{t_q}\in L^\infty(\Rr)$ and 
\[ \|f_{t_q}\|_\infty \le \frac{1}{2\pi}\left \| \widehat{\mu} \right\|_1 .\]
It remains to estimate the size of $\left \| \widehat{\mu} \right\|_1$ in terms of $q$ and $t_q$.

Define, for all $\xi\in \Rr$, \[ F(\xi,\chi)=\prod_{\gamma_\chi>0}\J\left(\frac{2|\xi|}{\sqrt{\tfrac{1}{4}+\gamma_\chi^2}}\right)\]
so that \[ \widehat{\mu}(\xi)=\exp\bigl(-i\langle t_q,r_q\rangle \xi\bigr)\prod_{\chi \ne \chi_0}F\left( \widehat{t_q}(\chi) \xi ,\chi\right).\]
It turns out that under some assumptions on $\widehat{t_q}$, the size of $\left \| \widehat{\mu} \right\|_1$ is uniformly bounded:
\begin{lemma}
Assume $\mathrm{GRH}$ and $\mathrm{LI}$. Let $\varepsilon>0$ and assume that $\|\widehat{t_q}\|_\infty \ge \varepsilon$, we have $\left \| \widehat{\mu} \right\|_1\ll_\varepsilon 1$.
\end{lemma}
\begin{proof}
By \cite{FiM}*{Lemma 2.16} there exists an absolute constant $c>0$ such that for all $|x|\ge 200$ and for all $\chi\ne \chi_0$ we have $|F(x,\chi)F(x,\overline{\chi})|\le \exp(-c|x|)$. Let $\psi$ be such that $\|\widehat{t_q}\|_\infty=|\langle t_q,\psi \rangle|$. Note that $\psi \neq \chi_0$ since we assumed $\langle t_q, \chi_0 \rangle = 0$. For all $|x|\ge \frac{200}{\varepsilon}$ we have $|F(x\langle t_q,\psi \rangle, \psi)F(x\langle t_q,\overline{\psi} \rangle, \overline{\psi})|\le \exp(-c\varepsilon |x|) $. Two cases arise: if $\psi$ is real, then $$|F(x\langle t_q,\psi \rangle, \psi)| \le \exp(-c\varepsilon |x|/2).$$
We use the fact that $|\J(\xi)| \le 1$ for all $\xi \in \Rr$ to deduce that in this case we have, for all $x\ge \frac{200}{\varepsilon}$, $$|\widehat{\mu}(x)|\le \exp\left(-\frac{c\varepsilon}{2}|x|\right)\, .$$
Otherwise, if $\psi\ne \overline{\psi}$, we deduce similarly that for all $|x|\ge \frac{200}{\varepsilon}$ we have $$|\widehat{\mu}(x)|\le \exp\left(- c\varepsilon|x|\right)\, .$$
Hence \[ \left\|\widehat{\mu} \right\|_1 \le \frac{400}{\varepsilon}+\int_{|x|\ge 200/\varepsilon}\exp\left(-\frac{c\varepsilon}{2}|x|\right)\mathrm{d}x \ll_\varepsilon 1 .\]

\end{proof}

\begin{theorem}\label{density-convergence}
Assume $\mathrm{GRH}$ and $\mathrm{ELI_A}(t_q)$ for some $A>1$. Let $q\ge 3$ and assume that there exist $\varepsilon>0$ such that $\|\widehat{t_q}\|_\infty\ge \varepsilon$. Then, uniformly for $$X\ge (|\supp(\hat{t_q})| \log k(t_q) )^{(\mathcal{L} |\supp(\hat{t_q})| \log k(t_q)) ^A}, $$ 
where $\mathcal{L}>0$ is a sufficiently large constant, we have
    \[\left| \delta_{t_q} -\frac{1}{X-\log 2}\int_{\log 2}^X \mathds{1}_{(0,\infty)}\left(E(y)\right)\mathrm{d}y \right| \ll_{A,\varepsilon}  C(t_q)^{\tfrac{3}{4}}  \log \bigl(C(t_q) \bigr)^{\tfrac{1}{4A}}   (\log \log X) (\log X) ^{-\tfrac{1}{4A}}  .\]
\end{theorem}
\begin{proof}
It suffices to combine Lemma \ref{density-wass} with Theorem~\ref{quantitative-limdist} together with the fact that \[\|f_{t_q}\|_\infty \ll \| \widehat{\mu} \|_1\ll_\varepsilon 1 .\]
\end{proof}

If we apply this result to $t_q=\varphi(q)\bigl(\mathds{1}_{\{a\}}-\mathds{1}_{\{b\}}\bigr)$, we obtain the quantitative rate announced in the introduction for the race between two residue classes $a \mods q$ and $b \mods q$:

\begin{proof}[Proof of Theorem $\ref{cor: effective rate a vs b}$]
Define 
\[ M_q:=\{ \chi \ne \chi_0\, :\, |\chi(a)-\chi(b)| \ge 1\}\, .\]
Parseval's identity applied to $t_q=\varphi(q)\bigl(\mathds{1}_{\{a\}}-\mathds{1}_{\{b\}}\bigr)$ shows that \[ 2\varphi(q)=\sum_\chi |\chi(a)-\chi(b)|^2 \le \sum_{\chi \notin M_q}1+\sum_{\chi \in M_q}4=\varphi(q)-\bigl|M_q\bigr|+4\bigl|M_q\bigr| \, .\]
Therefore  $\bigl|M_q\bigr| \ge \varphi(q)/3$. In particular, $\|\widehat{t_q}\|_\infty\ge 1$. Using \eqref{eq: def k(t_q)} and \eqref{eq:Ctq} we immediately get $|\supp(\hat{t_q})| \log k(t_q)\le \varphi(q) \log q$ and $C(t_q)\ll \varphi(q)^2 \log q$, and Theorem~\ref{density-convergence} implies the result.
\end{proof}

Theorem \ref{density-convergence} also applies to the context of the race between quadratic residues and nonresidues, where it yields the following quantitative rate of convergence:

\begin{corollary}\label{cor-convdensrvsnr}
Assume $\mathrm{GRH}$ and $\mathrm{ELI_A}(\ell_q)$ for some $A>1$. Let $q\ge 3$ and let $$X\ge (\rho (q) \log \rad(q) )^{(\mathcal{L} \rho (q) \log \rad(q)) ^A}, $$ 
where $\mathcal{L}>0$ is a sufficiently large constant, we have
    \[\left| \delta_{\ell_q} -\frac{1}{X - \log 2}\int_{\log 2}^X \mathds{1}_{(0,\infty)}\left(\pi(e^y;\ell_q)\right)\mathrm{d}y \right| \ll_{A} \rho(q)^{\tfrac{3}{4}} (\log \rad(q))^{\tfrac{3}{2}}\log (\rho(q)\log \rad(q))^{\tfrac{1}{4A}}\frac{(\log \log X)}{ (\log X) ^{\tfrac{1}{4A}}}  .\]
\end{corollary}
\begin{proof}
Using orthogonality relations of quadratic characters, we have 
\[ \ell_q=\sum_{\substack{\chi^2=\chi_0\\\chi\ne \chi_0}}\chi\, .\]
This proves in particular that $\|\widehat{\ell_q}\|_\infty\geq 1$.
Since every real character $\chi$ modulo $q$ is induced by a character $\chi_d$ modulo $d_q$, where \[ d_q=\begin{cases}
\rad(q)\quad &\text{if}\quad 4\text{ does not divide }q,\\
2\,\rad(q)\quad &\text{if}\quad 4||q,\\
4\, \rad(q)\quad &\text{if}\quad 8 \text{ divides } q
\end{cases}
\] and in this case $\chi=\chi_d$, we have \[ \left| \delta_{\ell_q} -\frac{1}{X-\log 2}\int_{\log 2}^X \mathds{1}_{(0,\infty)}\left(\pi(e^y;\ell_q)\right)\mathrm{d}y \right|=\left| \delta_{\ell_{d_q}} -\frac{1}{X-\log 2}\int_{\log 2}^X \mathds{1}_{(0,\infty)}\left(\pi(e^y;\ell_{d_q})\right)\mathrm{d}y \right|.\]
Thus, we may assume that $q$ is of the form $2^e q'$ where $q'$ is an odd squarefree integer and $e\in\{0,1,2,3\}$. Again, using \eqref{eq: def k(t_q)} and \eqref{eq:Ctq}, we obtain $|\supp(\hat{t_q})| \log k(t_q)\leq \rho(q) \log q$, and $C(\ell_q)\leq \rho(q) (\log q)^2$, and the result follows from Theorem~\ref{density-convergence}.
\end{proof}

\subsection{Application to generalized Skewes' numbers} \label{sec:Skewes}

In the case of prime number races, as explained by Fiorilli~\cite{Fiorilli}, one might consider $t_q$ with small $\delta_{t_q}$, and define its associated Skewes' number as \[ x(t_q):=\inf\left\{x\ge 2\, :\, \pi(x;t_q) >0 \right\} .\]
A natural question is: what is the rate of growth of $x(t_q)$ in terms of $q$? 

A first instance where we can answer this question is the classical prime number race, corresponding to $t_q=\varphi(q)\bigl(\mathds{1}_{\{a\}}-\mathds{1}_{\{b\}}\bigr)$. In this case, we denoted the Skewes' number by $x_{q;a,b}$ and stated an upper bound for this number in Theorem \ref{skewes-primerace} of the introduction, which we prove now.

\begin{proof}[Proof of Theorem $\ref{skewes-primerace}$]
By Theorem \ref{cor: effective rate a vs b} and the fact that $A >1$, we have     
\[\left| \delta_{t_q} -\frac{1}{X-\log 2}\int_{\log 2}^X \mathds{1}_{(0,\infty)}\left(E(y)\right)\mathrm{d}y \right| \ll_{A} \varphi(q)^{\tfrac{3}{2}} \log q\,(\log \log X) (\log X) ^{-\tfrac{1}{4A}}  .\] 
But by \cite[Theorem 1.11]{FiM}, we also know that $\delta_{t_q} \geq C = 10^{-5}$, an absolute constant independent of $a, b, q$ and $A$ (the \enquote{worst case} being attained for $a=1, b=5$ and $q=24$). This yields $$\frac{1}{X-\log 2}\int_{\log 2}^X \mathds{1}_{(0,\infty)}\left(E(y)\right)\mathrm{d}y \geq C - K_A \varphi(q)^{\tfrac{3}{2}} \log q (\log \log X) (\log X) ^{-\tfrac{1}{4A}},$$ where $K_A > 0$ is the implicit constant in Theorem \ref{cor: effective rate a vs b}. Therefore, it suffices to choose an $X_0$ such that $\frac{(\log X_0)^{1/4A}}{\log \log X_0} \gg_A \varphi(q)^{3/2} \log q$ to obtain $\frac{1}{X-\log 2}\int_{\log 2}^{X_0} \mathds{1}_{(0,\infty)}\left(E(y)\right)\mathrm{d}y > 0$. This is the case as soon as $\log \log X_0 = C_A \log \varphi(q)$ for some other constant $C_A > 0$ depending on $A$. In particular, there must exist a $y \leq X_0$ such that $\pi(e^y; t_q) > 0$, \emph{i.e.} $\log \log \log x_{q;a,b} \leq \log \log y \leq \log \log X_0 \ll_A \log \varphi(q)$.
\end{proof}

The case of $x_{q;R,NR}$ (corresponding to $t_q = \ell_q$) requires more work as $\delta_{\ell_q}$ might go to zero. Fiorilli showed, using large deviations results of Montgomery-Odlyzko~\cite{Mo-Od}, that (assuming GRH and LI for Dirichlet $L$-functions modulo $q$) this happens if and only if $\frac{\rho(q)}{\log \rad(q)} \to \infty$ as $q\to \infty$. 
Let us start by going back to the general setting ($t_q$ arbitrary). Following standard results (see for instance Fiorilli-Jouve~\cite{FJ}*{Proposition 3.18}) one proves that, under $\mathrm{GRH}$ and LI, $\mu$ has mean $\E(t_q)=-\langle t_q,r_q\rangle$ and its variance is given by \[ \Var(t_q)=2\sum_{\chi\ne \chi_0} |\langle t_q,\chi \rangle |^2 \sum_{\gamma_\chi>0}\frac{1}{\tfrac{1}{4}+\gamma_\chi^2}.\]
An important quantity in what follows is \[B(t_q):=\frac{\E(t_q)}{\sqrt{\Var(t_q)}} .\]

\begin{theorem}\label{skewes-result}
      Assume $\mathrm{GRH}$ and $\mathrm{ELI_A}(t_q)$ for some $A>1$. Assume that for all sufficiently large $q$ and all Dirichlet characters $\chi$ modulo $q$ we have \[ \widehat{t_q}(\chi)\in \{-1,0,1\}.\] If $B(t_q) \to -\infty$, then 
      \[ \log \log \log x(t_q) \ll_{A} B(t_q)^2+\log C(t_q)  .\]
\end{theorem}
In order to prove Theorem~\ref{skewes-result}, we will need the following Lemma: 
\begin{lemma}\label{density-bound}
Assume $\mathrm{GRH}$ and LI and assume that for all $q\ge 3$ and all Dirichlet character $\chi$ modulo $q$ we have \[ \widehat{t_q}(\chi)\in \{-1,0,1\}.\] If $E(t_q)\le 0$ then there exists an absolute constant $c_1$ such that \[ \delta_{t_q} \gg \exp \left(-c_1 B(t_q)^2 \right) \, .\] 
\end{lemma}
\begin{proof}
    This is a consequence of~\cite{FJ}*{Proposition 5.3}, which is itself a consequence of the large deviations results of Montgomery-Odlyzko~\cite{Mo-Od}: we have $\E(-t_q)\ge 0$, by~\cite{FJ}*{Proposition 5.3} there exists an absolute constant $c_1>0$ such that \[ 1-\delta_{-t_q} \gg \exp(-c_1 B(t_q)^2) .\]
    The result follows from the equality $\delta_{t_q}=\mu(0,\infty)=1-\delta_{-t_q}$.
\end{proof}
\begin{proof}[Proof of Theorem~$\ref{skewes-result}$]
We determine $X_0(q)$ for which \[ \frac{1}{X_0(q)}\int_{\log 2}^{X_0(q)}\mathds{1}_{(0,\infty)}(E(y)) \mathrm{d}y >0.\]
This means that there exists $x_0(q) =e^y \le \exp(X_0(q))$ such that $\pi(x_0(q);t_q)>0$. Thus, 
\[ \log \log \log x(t_q) \le \log \log \log x_0(q) \leq \log \log X_0(q) \, .\]
By Theorem~\ref{density-convergence} we have for all $$X\ge (|\supp(\hat{t_q})| \log k(t_q) )^{(\mathcal{L} |\supp(\hat{t_q})| \log k(t_q)) ^A}, $$ 
where $\mathcal{L}$ is a sufficiently large constant, we have
\[ \left| \delta_{t_q}-\frac{1}{X-\log 2}\int_{\log 2}^X \mathds{1}_{(0,\infty)}(E(y)) \mathrm{d}y \right| \ll_{A} \frac{C(t_q)}{(\log X)^{\tfrac{1}{5A}}} .\]
Thus, it suffices to choose $X_0(q)$ so that \[ \log X_0(q) = \left(K_{A} \frac{C(t_q)}{\delta_{t_q}}\right)^{5A},\]
where $K_{A}$ is a large constant depending only on $A$. We note that our condition on $t_q$ implies that $\lambda(t_q)=|\supp(\hat{t_q})|$, thus $C(t_q)>|\supp(t_q)|\log k(t_q)$. Hence $X_0(q)$ satisfies the condition $X_0(q)\ge (|\supp(\hat{t_q})| \log k(t_q) )^{(\mathcal{L} |\supp(\hat{t_q})| \log k(t_q)) ^A}$. 
By Lemma~\ref{density-bound}, we deduce that \[ \log \log X_0(q) \ll_{A} B(t_q)^2+\log C(t_q). \]
The result follows.
\end{proof}

Finally, we can prove Theorem \ref{skewes-qrnr} on the Skewes' number for the race between quadratic residues and nonresidues.

\begin{proof}[Proof of Theorem $\ref{skewes-qrnr}$]
    As in the proof of Corollary~\ref{cor-convdensrvsnr}, we may assume that, for all $n\ge1$, $q_n$ is of the form $2^e q_n'$ where $q_n'$ is an odd squarefree integer and $e\in\{0,1,2,3\}$. Since \[ \ell_{q_n}=\sum_{\substack{\chi^2 =\chi_0\\ \chi \ne \chi_0}}\chi ,\]
    then for all $\chi$ modulo $q_n$ we have $\langle \ell_{q_n} ,\chi\rangle\in\{0,1\}$. Moreover, Fiorilli~\cite{Fiorilli} proved that \[ B(\ell_{q_n}) \asymp -\sqrt{\frac{\rho(q_n)}{\log \rad(q_n)}} \, . \]
    Since $C(\ell_{q_n})\ll \rho(q_n) (\log \rad(q_n))^2$ and $\log \rad(q_n)\ll \rho(q_n)$, we have $\log C(\ell_{q_n})\ll \log \rho(q_n)$, thus, it suffices to apply Theorem~\ref{skewes-result} to deduce the theorem.
\end{proof}

\section*{Appendix: Proof of Corollary~\ref{sums-zeros}} \label{appendix_proof_sum_zeros}
Let $\chi, \lambda$ be Dirichlet characters modulo $q\ge3$.
We want to give an upper bound to the following sum: 
\[ \sum_{\substack{|\gamma_\chi|\ge T,\\|\gamma_\lambda| \ge T}} \frac{1}{|\gamma_\chi \gamma_\lambda|}\min\left( Y, \frac{1}{|\gamma_\chi-\gamma_\lambda|}\right)\]
in terms of $T,Y\ge 2$ and $q$. Using symmetry of zeros it suffices to consider the sum when $\gamma_\chi$ runs over positive imaginary parts and allowing $\gamma_\lambda$ to take positive and negative values. Noticing that 
\begin{equation}
    \frac{1}{\frac{1}{Y}+|\gamma_\chi-\gamma_\lambda|}\le \min\left( Y, \frac{1}{|\gamma_\chi-\gamma_\lambda|}\right) \le \frac{2}{\frac{1}{Y}+|\gamma_\chi-\gamma_\lambda|} \, ,
\end{equation}
the problem is thus equivalent to estimating the sum 
\[ \sum_{\substack{\gamma_\chi\ge T,\\|\gamma_\lambda| \ge T}} \frac{1}{|\gamma_\chi \gamma_\lambda|\left(\frac{1}{Y}+|\gamma_\chi-\gamma_\lambda|\right)} \, .\]
Define 
\[ S(\chi,\lambda;T,Y) = \sum_{\substack{\gamma_\chi \ge T \\ |\gamma_\lambda| \ge T}}\frac{1}{|\gamma_\chi \gamma_\lambda|\left(1+Y|\gamma_\chi-\gamma_\lambda|\right)}\, .\]
\begin{proof}[Proof of Corollary~$\ref{sums-zeros}$]
We decompose 
\[S(\chi,\lambda;T,Y) \le \sum_{\gamma_\chi \ge T} \sum_{|\gamma_\lambda| \ge 1}\frac{1}{|\gamma_\chi \gamma_\lambda|\left(1+Y|\gamma_\chi-\gamma_\lambda|\right)}=\sum_{i=1}^5 S_i(T,Y)\]
where $S_1(T,Y),\dots,S_5(T,Y)$ are the sums over $\gamma_\chi\ge T$ corresponding respectively to the ranges of $\gamma_\lambda$ given by: $\gamma_\lambda \le -\gamma_\chi/2$, $|\gamma_\lambda|<\gamma_\chi/2$, $|\gamma_\lambda-\gamma_\chi| \le 1$, $1<|\gamma_\chi-\gamma_\lambda|\le\gamma_\chi/2$, $\gamma_\lambda> 3\gamma_\chi/2$.\\
\textbf{Estimating $S_1(T,Y)$:}
If $\gamma_\lambda \le -\frac{\gamma_\chi}{2}$, we have $1+Y|\gamma_\chi-\gamma_\lambda|=1+Y\gamma_\chi+Y|\gamma_\lambda|>Y|\gamma_\lambda|$.
Thus \[ S_1(T,Y) = \sum_{\gamma_\chi \ge T} \sum_{\gamma_\lambda \le -\gamma_\chi/2}\frac{1}{|\gamma_\chi \gamma_\lambda|\left(1+Y|\gamma_\chi-\gamma_\lambda|\right)} \le \sum_{\gamma_\chi \ge T} \sum_{\gamma_\lambda \le -\gamma_\chi/2}\frac{1}{Y\gamma_\chi \gamma_\lambda^2}\, . \]
Since \[  \sum_{\gamma_\lambda \le -\gamma_\chi/2} \frac{1}{\gamma_\lambda^2} \ll \frac{\log q \gamma_\chi}{\gamma_\chi}\, ,\]
we have \[ S_1(T,Y) \ll \sum_{\gamma_\chi \ge T} \frac{ \log q\gamma_\chi }{Y\gamma_\chi^2}\ll \frac{(\log qT)^2}{YT}\, .\]
\textbf{Estimating $S_2(T,Y)$:} If $|\gamma_\lambda|\le \gamma_\chi/2$ then $1+Y|\gamma_\chi-\gamma_\lambda| = 1+Y(\frac{\gamma_\chi}{2}-\gamma_\lambda)+\frac{Y\gamma_\chi}{2} >\frac{Y\gamma_\chi}{2}$. \\
Moreover, \[ \sum_{1\le |\gamma_\lambda |<\gamma_\chi/2} \frac{1}{|\gamma_\lambda|} \ll \sum_{1\le n\le \gamma_\chi} \frac{\log qn}{n}\ll (\log q \gamma_\chi) \log \gamma_\chi \, .\]
Thus, \[S_2(T,Y)=\sum_{\gamma_\chi \ge T} \sum_{1\le |\gamma_\lambda |<\gamma_\chi/2}\frac{1}{|\gamma_\chi \gamma_\lambda|\left(1+Y|\gamma_\chi-\gamma_\lambda|\right)}\ll \sum_{\gamma_\chi \ge T}(\log q)\frac{(\log \gamma_\chi)^2}{Y \gamma_\chi^2}\ll (\log q)^2\frac{(\log T)^3}{YT}\, .\]
\textbf{Estimating $S_3(T,Y)$:} If $|\gamma_\lambda-\gamma_\chi | \le 1$, we use the trivial bound $1+Y|\gamma_\lambda-\gamma_\chi| \ge 1$ and $|\gamma_\lambda|\gg \gamma_\chi$. Thus \[S_3(T,Y)=\sum_{\gamma_\chi \ge T} \sum_{\substack{\gamma_\lambda \\ |\gamma_\lambda-\gamma_\chi| \le 1}}\frac{1}{|\gamma_\chi \gamma_\lambda|\left(1+Y|\gamma_\chi-\gamma_\lambda|\right)} \ll \sum_{\gamma_\chi \ge T} \frac{\log q\gamma_\chi}{\gamma_\chi^2}\ll \frac{(\log qT)^2}{T}\, .\]
\textbf{Estimating $S_4(T,Y)$:} When $1<|\gamma_\lambda-\gamma_\chi|\le \gamma_\chi/2$, we use the bound 
\[ \sum_{\substack{\gamma_\lambda \\ 1\le |\gamma_\lambda -\gamma_\chi|\le \gamma_\chi/2}} \frac{1}{1+Y|\gamma_\chi-\gamma_\lambda|} \ll \frac{\log q \gamma_\chi}{Y}\sum_{1\le n \le \gamma_\chi }\frac{1}{n} \ll \frac{(\log q \gamma_\chi)\log \gamma_\chi}{Y}\, .\]
Thus, \[S_4(T,Y)\ll \sum_{\gamma_\chi \ge T}(\log q)\frac{(\log \gamma_\chi)^2}{\gamma_\chi ^2} \ll (\log q)^2\frac{(\log T)^3}{YT}\]
\textbf{Estimating $S_5(T,Y)$: }If $\gamma_\lambda \ge 3\gamma_\chi/2$ we have \[Y|\gamma_\chi-\gamma_\lambda|+1=Y\left(\frac{2\gamma_\lambda}{3}-\gamma_\chi\right)+1+\frac{Y\gamma_\lambda}{3}\ge \frac{Y\gamma_\lambda}{3}\, .\]
Thus, \[S_5(T,Y) \ll \sum_{\gamma_\chi \ge T} \sum_{\gamma_\lambda \ge 3\gamma_\chi/2} \frac{1}{Y\gamma_\chi \gamma_\lambda^2} \ll \sum_{\gamma_\chi \ge T} \frac{\log q\gamma_\chi}{Y\gamma_\chi ^2} \ll \frac{(\log q T)^2}{YT}\, .\]
This proves that \[ S(\chi,\lambda; T,Y) \ll (\log q)^2\left( \frac{(\log T)^2}{T}+\frac{(\log T)^3}{YT}\right)\, .\]
The corollary is thus deduced by multiplying by $Y$. 
\end{proof}

\bibliography{bibliochebyshev}

\end{document}